\documentclass{amsproc}

\usepackage{amsmath,amsfonts,amssymb,amsthm}
\usepackage{epsfig}
\usepackage[all]{xy}
\newtheorem{theorem}{Theorem}[section]
\newtheorem{lemm}[theorem]{Lemma}
\newtheorem{prop}[theorem]{Proposition}
\newtheorem{theo}[theorem]{Theorem}
\theoremstyle{definition}
\newtheorem{defi}[theorem]{Definition}

\newtheorem{coro}[theorem]{Corollary}
\theoremstyle{remark}
\newtheorem{remark}[theorem]{Remark}

\numberwithin{equation}{section}

\renewcommand{\a}{\alpha}
\newcommand{\ch}{\mbox{ch}}
\newcommand{\Cs}{\mathbb C^{\times} }

\newcommand{\FGC}{{\mathcal F}_{\G\times \mathbb C^{\times}\times \mathbb C^{\times}} }

\newcommand{\FGGC}{\overline{\mathcal F}_{\G\times \mathbb C^{\times}\times \mathbb C^{\times}}}

\newcommand{\G}{\Gamma}
\newcommand{\GC}{\Gamma\times \mathbb C^{\times}}

\newcommand{\Gm}{{\Gamma}_m}
\newcommand{\GCm}{{\Gamma}_m{\times}\mathbb C^{\times}{\times}\mathbb C^{\times}}
\newcommand{\Gn}{{\Gamma}_n}
\newcommand{\GCn}{{\Gamma}_n{\times}\mathbb C^{\times}{\times}\mathbb C^{\times}}
\newcommand{\g}{\gamma}
\newcommand{\hg}{\widehat{\mathfrak h}_{\G, \wt}}

\newcommand{\loopg}{\widehat{\mathfrak g}}

\newcommand{\LG}{ {\Lambda}_{\G}}

\newcommand{\ep}{\epsilon}

\newcommand{\RG}{R_{\G}}

\newcommand{\RGGC}{\overline{R}_{ \G\times \mathbb C^{\times}\times \mathbb C^{\times}}}

\newcommand{\RGC}{R_{\G{\times}\mathbb C^{\times}{\times}\mathbb C^{\times}}}

\newcommand{\Rz}{R_{\mathbb Z}(\Gamma)}
\newcommand{\Rzz}{\overline{R}_{\mathbb Z}({\G})}
\newcommand{\s}{\sigma}
\newcommand{\SG}{ S_\G }
\newcommand{\SGC}{S_{\G{\times}\mathbb C^{\times}{\times}\mathbb C^{\times}}}

\newcommand{\SGGC}{\overline{S}_{\G\times\mathbb C^{\times} }}

\newcommand{\VG}{V_{ \G}}
\newcommand{\VGC}{V_{\G\times \mathbb C^{\times}\times \mathbb C^{\times}}}

\newcommand{\wt}{\xi}

\def\om{\omega}

\begin{document}

\title[Two-parameter quantum McKay correspondence]
    {Two-parameter quantum vertex representations via finite groups and
the McKay correspondence}

\author[Jing]{Naihuan Jing}
\address{NJ: Department of Mathematics,
   North Carolina State University,
   Ra\-leigh, NC 27695, USA}
\email{jing@math.ncsu.edu}

\author[Zhang]{Honglian Zhang$^{\star}$}
\address{HZ: Department of Mathematics,
Shanghai University, Shanghai 200444, China}
\email{hlzhangmath@shu.edu.cn}
\thanks{$^\star$H.Z., Corresponding Author}


\begin{abstract}
We introduce two-parameter quantum toroidal algebras of simply laced
types and provide their group theoretic realization using finite
subgroups of $SL_2(\mathbb C)$ via McKay correspondence. In
particular our construction contains a realization of the vertex
representation of the two-parameter quantum affine algebras of $ADE$
types.
\end{abstract}

\maketitle

\section{ Introduction}

In a series papers \cite{FJW1, FJW2} the basic representations of
two-toroidal Lie algebras and their quantum analogs, including
affine Lie algebras and quantum affine algebras of simply laced
types as subaglebras, were constructed from representation theory of
finite groups of $SL_2(\mathbb C)$ via the celebrated McKay
correspondence. Using purely representation theoretic data the
Frenkel-Kac \cite{FK} and Frenkel-Jing \cite{FJ} vertex
representations are constructed as a by-product of the unified group
theoretic constructions from the root lattice of the corresponding
finite dimensional Lie algebra $\mathfrak g$. In \cite{J4} we have
pointed out that such a uniformed construction incorporate not only
the toroidal Lie algebras, quantum toroidal algebras \cite{VV} (see
also \cite{J3}) but also give a general algebraic machinery to
realize other related algebraic structures.

In the current paper we provide a two-parameter quantum analog of
the toroidal Lie algebras of simply laced types using the new form
of McKay correspondence. In particular this gives a group theoretic
realization of the newly revitalized two-parameter quantum affine
algebras \cite{HRZ, Z} as distinguished subalgebras of our new
quantum toroidal algebras. As expected, our construction degenerates
to the quantum affine case by specializing $r=s^{-1}$. Through this
new construction we have also shown that the vertex representation
constructed in \cite{HZ, Z} is a natural generalization of
Frenkel-Jing construction and also reconfirm the two-parameter
generalization of the usual quantum affine algebras and their
Drinfeld realization. Our construction further shows that the
toroidal version is more suitable and natural in this picture and
reveal more symmetry in the structure of the two parameter quantum
toroidal algebras.

The new two-parameter quantum toroidal Lie algebras or double affine
algebras provide a new layer of generalization or quantization. To
show the relationships we can depict them in the following diagram
similar to that given by I. Frenkel earlier. Our new algebra adds a
new direction of generalization.

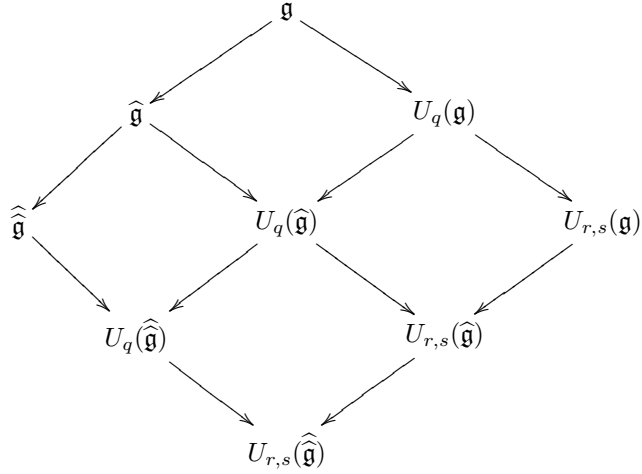
\begin{displaymath}
    \xymatrix@C=2pc@R=2pc{
    &   & \mathfrak g \ar[dl] \ar[dr] & &\\
&\widehat{\mathfrak g} \ar[dl]\ar[dr]     & & U_q({\mathfrak g})\ar[dl]\ar[dr] &\\
\widehat{\widehat{\mathfrak g}} \ar[dr]& & U_q(\widehat{\mathfrak
g}) \ar[dr]\ar[dl] & &
U_{r,s}({\mathfrak g})\ar[dl] \\
&U_q(\widehat{\widehat{\mathfrak g}})\ar[dr] & &
U_{r,s}(\widehat{\mathfrak
g}) \ar[dl]& \\
&& U_{r,s}(\widehat{\widehat{\mathfrak g}}) &&}
\end{displaymath}

It is amazing that each vertex in the diagram admits a realization
through McKay correspondence, and each level of complexity is
achieved by replacing the finite group $\Gamma$ by $\Gamma\times
\mathbb C^{\times}$ and by $\Gamma\times \mathbb C^{\times}\times
\mathbb C^{\times}$. The quantum parameter $q$ and $(r, s)$ are
respectively represented by special characters on the group $\mathbb
C^{\times}$ and $\mathbb C^{\times}\times\mathbb C^{\times}$.

At the early stage of the development of quantum groups it was
already realized there exist multi-parameter quantum groups \cite{R,
T}. Though two-parameter cases play an important role in the dual
quantum picture, the development of two-parameter quantum groups
only took a turn after more structures are revealed in series of
papers of Benkart and Witherspoon \cite{BW1, BW2, BW3} for type $A$.
Shortly after these generalizations to other types are given by
Bergeron, Gao and Hu \cite{BGH1, BGH2, BH, HS}. Recently Hu, Rosso
and one of us found the two parameter quantum affine algebras of ADE
types \cite{HRZ, Z} via their vertex representations \cite{HZ} based
on generalized Drinfeld realizations (cf. \cite{J2}), which are
two-parameter analog of the Frenkel-Jing representations \cite{FJ}.
In all these work one realizes that the two-parameter analog, like
its one-parameter case, amounts to a clever deformation of the
natural number $n$ to the two-parameter quantum number:
\begin{equation}
[n]=\frac{r^n-s^n}{r-s}=r^{n-1}+r^{n-2}s+\cdots+rs^{n-2}+s^{n-1}.
\end{equation}
Clearly when $rs=1$, the two-parameter quantum number degenerates to
the one-parameter quantum number.

As in the previous construction of McKay correspondence and quantum
toroidal algebras we recover the basic representation of $U_{r,
s}(\widehat{\mathfrak g})$ by choosing
$$\xi=\g_0\otimes ((rs^{-1})^{\frac{1}{2}}+
(r^{-1}s)^{\frac{1}{2}})-\pi\otimes 1_{\Cs{\times}\Cs},
$$
where $\gamma_0$, $1_{\Cs{\times}\Cs}$ are the trivial characters of
$\Gamma$ and $\mathbb C^{\times}\times \mathbb C^{\times}$
respectively, $r$ and $s$ are two independent natural characters of
$\mathbb C^{\times}$, and $\pi$ is the natural character of the
imbedding of $\Gamma$ in $SL_2(\mathbb C)$. The natural appearance
of the two parameter quantum toroidal algebra in this picture
intrinsically shows its importance in representation theory of two
parameter quantum affine algebras.

The paper is organized as follows. Section 2 recalls the basic
material of wreath products $\Gamma_n$ of symmetric groups
associated with any finite group and the Hopf algebra structures in
the representation ring. Section 3 studies the representation ring
$R(\Gamma_n\times\mathbb C^{\times}\times \mathbb C^{\times})$ and
weighted bilinear forms. Section 4 gives the two-parameter McKay
weights for each finite subgroup of $SL_2(\mathbb C)$. Section 5
defines the two-parameter Heisenberg algebra and realizes its
canonical representation using group theoretic data out of $\Gamma$.
Section 6 gives the Frobenius-type characteristic map between $\RGC$
and $\SGC$. Section 7 realizes the two-parameter quantum vertex
operators using irreducible characters of $\Gamma_n\times\mathbb
C^{\times}\times \mathbb C^{\times}$, and finally in Section 8 we
introduce two-parameter quantum toroidal algebras and provide their
realization via McKay correspondence, and in particular this also
provides a group theoretic realization of the basic representation
of the two-parameter quantum affine algebras.

\section{Wreath products and vertex representations} \label{sect_wreath}

\subsection{The wreath product $\Gn$}
  Let $\Gamma$ be a finite group and $n$ a non-negative integer.
The wreath product $\Gn$ is the semidirect product of the $n$-th direct product $\G^n=\G\times\cdots
\times\G$ and the symmetric group $S_n$:
$$
 \Gamma_n = \{(g, \sigma) | g=(g_1, \ldots, g_n)\in {\Gamma}^n,
\sigma\in S_n \}
$$
 with the group multiplication
$$
(g, \sigma)\cdot (h, \tau)=(g \, {\sigma} (h), \sigma \tau ) ,
$$
where $S_n$ acts on $\G^n$ by permuting the factors.

Let $\G_*$ be the set of conjugacy classes of $\G$ consisting of
$c^0=\{1\}$, $c^1$, $\dots$, $c^\ell$ and $\G^*$ be the set of
$\ell+1$ irreducible characters: $\g_0, \g_1, \dots, \g_\ell$, where
$\gamma_0$ is the trivial character of $\Gamma$. The order of the
centralizer of an element in the conjugacy class $c$ is denoted by
$\zeta_c$, so the order of the conjugacy class $c$ is
$|c|=|\G|/\zeta_c$, where $|\G|$ is the order of $\G$.

 A partition $\lambda=(\lambda_1, \lambda_2, \ldots, \lambda_\ell)$
is a decomposition of $n=|\lambda|=\lambda_1+\cdots+\lambda_l$ with
nonnegative integers: $\lambda_1\geq \dots \geq \lambda_\ell \geq
1$, where $\ell=\ell (\lambda )$ is called the {\em length} of the
partition $\lambda $ and $\lambda_i$ are called the {\em parts} of
$\lambda$. Another notation for $\lambda$ is
$$
\lambda=(1^{m_1}2^{m_2}\cdots)
$$
with $m_i$ being the multiplicity of parts equal to $i$ in
$\lambda$.
Denote by $\mathcal P$ the set of all partitions of integers and by $\mathcal P(S)$ the set of
all partition-valued functions on a set $S$.
The weight of a partition-valued function $\rho=(\rho(s))_{s\in S}$ is defined to be $\|\rho\|=\sum_{s\in S}|\rho(s)|$. We also denote
by $\mathcal P_n$ (resp. $\mathcal P_n(S)$) the subset
of $\mathcal P$ (resp. $\mathcal P(S)$) of partitions
with weight $n$.

It is well-known that the conjugacy classes of $\Gn$ are
parameterized by partition-valued functions on $\G_*$. Let $x=(g,
\s)\in {\Gamma}_n$, where $g=(g_1, \ldots, g_n) \in {\Gamma}^n$ and
$\s\in S_n$ is presented as a product of disjoint cycles. For each
cycle $(i_1 i_2 \cdots i_k)$ of $\sigma$, we define the {\em
cycle-product} element $g_{i_k} g_{i_{k -1}} \cdots g_{i_1} \in
\Gamma$, which is determined up to conjugacy in $\Gamma$ by $g$ and
the cycle. For any conjugacy class $c\in \G$ and each integer $i\geq
1$, the number of $i$-cycles in $\sigma$ whose cycle-product lies in
$c$ will be denoted by $m_i(c)$. This gives rise to a partition
$\rho(c)=(1^{m_1 (c)} 2^{m_2 (c)} \ldots )$ for $c \in \G_*$. Thus
we obtain a partition-valued function $\rho=( \rho (c))_{c \in \G_*}
\in {\mathcal P} ( \G_*)$ such that $\|\rho\|=\sum_{i, c} i
m_i(\rho(c)) =n$. This is called the {\em type} of the element $(g,
\sigma)$.  It is known \cite{M} that two elements in the same
conjugacy class have the same type and there exists a one-to-one
correspondence between the sets $(\Gn)_*$ and $\mathcal P_n(\G_*)$.
We will say that $\rho$ is the type of the conjugacy class of $\Gn$.

Given a class $c$ we denote by $c^{-1}$ the class
$\{x^{-1}| x\in c\}$.
For each $\rho\in \mathcal P(\G_*)$ we also associate the
partition-valued function
$$\overline{\rho}=(\rho(c^{-1}))_{c\in \G_*}.$$

Given a partition $\lambda = (1^{m_1} 2^{m_2} \ldots )$, we denote
by $z_{\lambda } = \prod_{i\geq 1}i^{m_i}m_i!$
 the order of the centralizer
of an element of cycle type $\lambda$ in $S_{|\lambda|}$. The order
of the centralizer of an element $x = (g, \sigma) \in {\Gamma}_n$ of
type $\rho=( \rho(c))_{ c \in \G_*}$ is given by
$$
Z_{\rho}=\prod_{c\in \G_*}z_{\rho(c)}\zeta_c^{l(\rho(c))}.
$$

\subsection{Grothendieck ring $R_{\G\times C^{\times}\times C^{\times}}$}
Let $\Rz$ be the $\mathbb Z$-lattice generated by $\g_i$, $i=0,
\dots, r$, and $R(\G)=\mathbb C\otimes\Rz$ be the space of complex
class functions on
 the group
$\G$. In the previous work on the McKay correspondence and vertex
representations \cite{W, FJW1}, the Grothendieck ring $
  \RG = \bigoplus_{n\geq 0} R({\Gamma}_n)$ was studied.
In  quantum case, the Grothendieck ring was $
  R_{\Gamma\times \mathbb C^{\times}} = \bigoplus_{n\geq 0} R({\Gamma}_n\times\mathbb
  C^{\times})$ \cite{FJW2}.
  In our two-parameter quantum case we
need to add the another ring $R(\mathbb C^{\times})$, the space of
characters of $\mathbb C^{\times} \times \mathbb
C^{\times}=\{(t_1,\,t_2)\in\mathbb C\times \mathbb C|t_1,\,t_2\neq
0\}$.

Let $r,\,s$ be the irreducible character of $\Cs$ that sends $t$ to
itself. Then $R(\mathbb C^{\times}\times \mathbb C^{\times})$ is
spanned by irreducible multiplicative characters $r^ms^n$, $m,\,n\in
\mathbb Z$, where
\[
r^m(t_1)=t_1^m, \qquad s^n(t_2)=t_2^n\qquad t_1,\,t_2\in \mathbb
C^{\times}.
\]
Thus $R(\mathbb C^{\times}\times \mathbb C^{\times})$ is identified
with the ring $\mathbb C[r^{\pm1}, s^{\pm1}]$, and we have
\[
R(\G\times \mathbb C^{\times}\times \mathbb C^{\times})=R(\G)\otimes
R(\mathbb C^{\times}\times \mathbb C^{\times}).
\]

An elements of $R(\G\times \mathbb C^{\times}\times \mathbb
C^{\times})$ can be written as a finite sum:
\[
f=\sum_if_i\otimes r^{m_i}s^{n_i}, \qquad f_i\in R(\G), m_i,\,
n_i\in\mathbb Z.
\]

We can also view $f$ as a function on $\G$ with values in the ring
of Laurent polynomials $\mathbb C[r^{\pm1}, s^{\pm1}]$. In this case
we will write $f^{r,\,s}$ to indicate the formal variables $r,\,s$,
then $f^{r,\,s}(c)=\sum_if_i(c)r^{m_i}s^{n_i}\in\mathbb C[r^{\pm1},
s^{\pm1}]$. As a function on $\G\times \mathbb C^{\times}\times
\mathbb C^{\times}$, we have $f(c, t_1,
t_2)=\sum_if_i(c)t_1^{m_i}t_2^{n_i}$.

Denote by $R_{\G\times C^{\times}\times C^{\times}}$ the following
direct sum:
\[
R_{\G\times C^{\times}\times C^{\times}}=\bigoplus_{n\geq 0}
R(\Gn\times \mathbb C^{\times}\times \mathbb C^{\times}) \simeq
R_{\G}\otimes \mathbb C[r^{\pm1}, s^{\pm1}].
\]

\subsection{Hopf algebra structure on $R_{\G\times C^{\times}\times C^{\times}}$}
       The multiplication $m$ in $\mathbb C^{\times}\times \mathbb C^{\times}$ and the diagonal map
$\mathbb C^{\times}\times \mathbb
C^{\times}\stackrel{d}{\longrightarrow}\mathbb C^{\times}\times
\mathbb C^{\times}\times \mathbb C^{\times}\times \mathbb
C^{\times}$ induce the Hopf algebra structure on $R(\mathbb
C^{\times}\times \mathbb C^{\times})$.
\begin{align}\label{E:hopf1}
m_{\mathbb C^{\times}\times C^{\times}}&: R(\mathbb C^{\times}{\times}
\mathbb C^{\times}){\otimes} R(\mathbb C^{\times}{\times} \mathbb
C^{\times})\stackrel{\cong }{\longrightarrow} R(\Cs{\times} \Cs{\times}
\Cs{\times}\Cs) 
 \stackrel{d^*}{\longrightarrow} R(\Cs{\times} \Cs),\\ \label{E:hopf2}
\Delta_{\Cs \times C^{\times}}&: R(\Cs)
\stackrel{m^*}{\longrightarrow}
 R(\Cs{\times} \Cs)
 \stackrel{\cong}{\longrightarrow}
 R(\Cs) \otimes R(\Cs).
\end{align}
In terms of the basis $\{r^ms^n\}$ we have
\begin{align*}
r^{m_1}s^{n_1}\cdot r^{m_2}s^{n_2}&=r^{m_1+m_2}s^{n_1+n_2},\\
\Delta(r^ms^n)&=r^ms^n\otimes r^ms^n,
\end{align*}
where we abbreviate $\Delta_{\Cs\times \Cs}$ by $\Delta$ and follow
the convention of writing $a\cdot b=m_{\Cs\times \Cs}(a\otimes b)$.

The antipode $S_{\Cs\times \Cs}$ and the counit $\epsilon_{\Cs\times
\Cs}$ are given by
\[
S_{\Cs\times \Cs}(r^ms^n)=r^{-m}s^{-n}, \qquad \epsilon_{\Cs\times
\Cs}(r^ms^n)=\delta_{m-n,0}.
\]

We extend the Hopf algebra structures on $R(\Cs{\times}\Cs)$ and
$R_{\G}$ into a Hopf algebra structure on
$R_{\G\times \Cs\times \Cs}$ using a standard procedure in Hopf algebra \cite{A}. The
multiplication and comultiplication are given by the respective
composition of the following maps:
\begin{align}
 m &: R(\GCn ) \otimes R(\GCm)
 \stackrel{\cong }{\longrightarrow} R(\GCn \times \GCm)\nonumber\\
 &\qquad \stackrel{1\otimes m_{\Cs{\times}\Cs}}{\longrightarrow}
R( {\Gn{\times}\Gm {\times}\Cs{\times}\Cs})\stackrel{Ind\otimes 1}{\longrightarrow}
R( {\Gamma}_{n + m}{\times}\Cs{\times}\Cs);\\
\Delta&: R(\GCn ) \stackrel{Res\otimes 1}{\longrightarrow}
 \oplus_{m=0}^nR( {\Gamma}_{n - m} \times \GCm)\nonumber\\
&\qquad \stackrel{1\otimes \Delta_{\Cs{\times}\Cs}}{\longrightarrow}
\oplus_{m=0}^n R( {\Gamma}_{n - m}\times \GCm{\times}\Cs{\times}\Cs)\nonumber\\
&\qquad \stackrel{\cong }{\longrightarrow}
 \oplus_{m=0}^nR( {\Gamma}_{n - m}{\times}\Cs{\times}\Cs) \otimes R(\GCm),
\label{E:comult}
\end{align}
where we have used the identification of $R(\Cs\times \Cs)$ with
$R(\Cs)\otimes R(\Cs)$ in (\ref{E:hopf1}-\ref{E:hopf2}). Also
$Ind: R(\Gn \times \Gm ) \longrightarrow R(\G_{n +m})$
denotes the induction functor
and $Res: R(\Gn ) \longrightarrow
 R( {\G }_{n - m} \times \Gm )$ denotes the restriction functor.

The antipode is given by
\[
S(f(g, (t_1, t_2)))=f(g^{-1}, (t_1^{-1}, t_2^{-1})), \qquad g\in \G, t\in \Cs{\times}\Cs.
\]
In particular, $S(\g)(c)=\g(c^{-1})$ for
$\g\in\G^*$. As we mentioned earlier, we may write
$f\in \RGC$ as
$$f^{r,\,s}(g)=\sum_if_i(g)r^{m_i}s^{n_i},$$
Then $S(f^{r,\,s})(g)=\sum_if_i(g^{-1})r^{-m_i}s^{-n_i}$.

The counit $\epsilon$ is defined by
\[
\epsilon(R(\GCn ))=0, \qquad\mbox{if \ \ } n\neq 0,
\]
and $\epsilon$ on $R(\Cs{\times}\Cs)$ is the counit of the Hopf algebra
$R(\Cs{\times}\Cs)$.

\section{A weighted bilinear form on $R(\GCn)$}
\label{sect_weight}

\subsection{A standard bilinear form on $\RGC$}
Let $f, g\in \RGC$ with $f=\sum_i f_i\otimes r^{m_i}s^{n_i}$ and
$g=\sum_i g_i\otimes r^{k_i}s^{l_i}$.
The $\mathbb C[r^{\pm 1}, s^{\pm 1}]$-valued standard $\mathbb C$-bilinear form
on $\RGC$ is defined as
\begin{align*}
\langle f, g \rangle_{\G}^{r\, s}
&= \sum_{i,j} \langle f_i, g_j\rangle_{\G}r^{m_i-k_j}s^{n_i-l_j}\\
&= \sum_{i,j}\sum_{c\in \G_*}\zeta_c^{-1}f_i(c)g_j(c^{-1})r^{m_i-k_j}s^{n_i-l_j},
\end{align*}
where we recall that
$c^{ -1}$ denotes the conjugacy class
$\{ x^{ -1}| x \in c \}$ of $\G$, and $\zeta_{c}$ is the order of the
centralizer of the class $c$ in $\G$. Sometimes we will also view the bilinear form as a function
of $(t_1, t_2)\in\Cs{\times}\Cs$:
\begin{equation*}
\langle f, g \rangle_{\G}^{r,\, s}(t_1, t_2) = \sum_{c \in \Gamma_*}
\zeta_c^{ -1} f(c, (t_1, t_2)) S(g(c, (t_1, t_2))).
\end{equation*}

The following is a direct consequence of the orthogonality
of irreducible characters of $\G$.
\begin{eqnarray}
  \langle \g_i\otimes r^m s^{n}, \g_j \otimes r^k s^{l}\rangle_{\G}^{r,\,s} &= & \delta_{ij}
r^{m-k}s^{n-l} , \nonumber \\
  \sum_{ \g \in \G^*} \g (c ') S(\g)( c)
    &= & \delta_{c, c '} \zeta_c, \quad c, c ' \in \G_*.  \label{eq_orth}
\end{eqnarray}

Let $\langle \ \ , \ \ \rangle_{\Gn}^{r,\, s}$ be the $\mathbb
C[r^{\pm 1}, s^{\pm 1}]$-valued bilinear form on $R(\GCn)$. The
$\mathbb C[r^{\pm 1}, s^{\pm 1}]$-valued standard bilinear form in
$\RGC$ is defined in terms of the bilinear form on $R( \GCn )$ as
follows:
\[
\langle u, v \rangle^{r,\, s}
 = \sum_{ n \geq 0} \langle u_n, v_n \rangle_{\Gn}^{r,\, s},
\]
where
$u = \sum_n u_n$ and $v = \sum_n v_n$ with $u_n, v_n\in R(\Gn{\times}\Cs{\times}\Cs)$.

\subsection{A weighted bilinear form on $\RGC$}
A class function $\wt \in\RGC$ is called {\t self-dual} if for all
$x\in\G, (t_1, t_2)\in\Cs{\times} \Cs$
\[
\xi(x, (t_1, t_2))=S(\xi(x, (t_1, t_2))),
\]
or equivalently $\xi^{r,\, s}(x)=\xi^{r^{-1},\, s^{-1}}(x^{-1})$.

We fix a self-dual class function $\xi$.
The tensor product of two representations $\g$ and $\beta$ in $ \RGC)$
will be denoted by $\g*\beta $.

Let $a_{ij} \in \mathbb C[r^{\pm 1}, s^{\pm 1}]$ be the (virtual) multiplicity
of $\g_j$ in $ \wt * \g_i $, i.e.,
\begin{eqnarray}  \label{eq_tens}
 \wt  *  \g_i
  = \sum_{j =0}^r a_{ij} \g_j.
\end{eqnarray}
We denote by $A^{r,\, s}$ the $ n \times n$ matrix $ ( a_{ij})_{0
\leq i,j \leq n-1}$.

Associated to $\xi$ we introduce the following weighted bilinear form
$$
  \langle f, g \rangle_{\wt}^{r,\, s} = \langle \wt * f ,  g \rangle_{\G}^{r,\, s},
   \quad f, g \in \RGC.
$$
where we use the superscript $r \, s$ to indicate the $r, s$-dependence.
The superscript $r \, s$ is often omitted if
the $r, s$-variable in characters $f$ and $g$ is clear from the context.
The explicit formula of the bilinear form is given as follows.
\begin{eqnarray}
  \langle f, g \rangle_{\wt}^{r,\,s}
   &=&\frac 1{ |\G|} \sum_{x\in \G}\wt^{r,\,s}(x)f^{r,\,s}(x)g^{r^{-1},\,s^{-1}}(x^{-1})
\nonumber\\
& =& \sum_{c \in \G_*} \zeta_c^{ -1} \wt^{r,\,s}(c) f^{r,\,s}(c)
g^{r^{-1},\,s^{-1}}(c^{ -1}),
     \label{eq_twist}
\end{eqnarray}
which is the  average of
the character $ \wt * f * \overline{g}$ over $\G$.

The self-duality of $\wt$ together with (\ref{eq_twist})
implies that
$$
 a_{ij} = \overline{a_{ji}},
$$
i.e. $A^{r\, s}$ is a hermitian-like matrix with the bar action
given by $\overline{r}=s, \overline{s}=r$.

The orthogonality (\ref{eq_orth}) implies that
\begin{equation}\label{E:qcartan}
a_{ij}=\langle \g_i, \g_j\rangle_{\xi}^{r,\,s}.
\end{equation}

\begin{remark}
 If $\wt$ is the trivial character $\g_0$, then the
 weighted bilinear form becomes the standard one on $ \RGC$.
\end{remark}

\subsection{A weighted bilinear form on $R( \GCn)$}

Let $V$ be a $\G{\times}\Cs{\times}\Cs$-module which affords a character $\g$ in $\RGC$.
We can decompose $V$ as follows:
\begin{equation*}
V=\bigoplus_iV_i\otimes \mathbb C({k_i}, l_i),
\end{equation*}
where $V_i$ is a (virtual) $\G$-module in $R(\G)$ and
$\mathbb C(k_i, l_i)$ is the one dimensional $\Cs{\times}\Cs$-module
afforded by the character $r^{k_i}s^{l_i}$.

The $n$-th outer tensor product
$V^{ \otimes  n} $ of $V$ can be regarded naturally as
a representation of the wreath product
$(\G{\times}\Cs{\times}\Cs)_n$ via permutation of the factors
and the usual direct product action. More precisely, note that $\G{\times}\Cs{\times}\Cs$
can be viewed as a subgroup of $(\G{\times}\Cs{\times}\Cs)_n$ by the diagonal inclusion from
$\Cs{\times}\Cs$ to $(\Cs{\times}\Cs)^n$:
\[
\Gn\times\Cs{\times}\Cs\longrightarrow ({\G}^n\times{\Cs{\times}\Cs}^n)\rtimes S_n=
(\G\times\Cs{\times}\Cs)_n.
\]
This provides a
 natural $\Gn\times\Cs{\times}\Cs$-module structure on $V^{\otimes n}$. We
denote its character by $\eta_n ( \g )$. Explicitly we have
\begin{equation}\label{E:eta-action}
(g, \sigma, (t_1, t_2)).(v_1\otimes\cdots\otimes v_n)=
(g_{1}, (t_1, t_2))v_{\sigma^{-1}(1)}\otimes\cdots
\otimes (g_{n}, (t_1, t_2))v_{\sigma^{-1}(n)},
\end{equation}
where $g=(g_1, \ldots, g_n)\in \G^n$.

Let $\varepsilon_n$ be the (1-dimensional) sign representation
of $\Gn$ so that $\G^n$ acts trivially while
letting $S_n$ act as a sign representation.
We denote by $\varepsilon_n ( \g ) \in R(\GCn)$
the character of the tensor product
of $\varepsilon_n\otimes 1$ and $V^{\otimes n}$.

The weighted bilinear form on $R( \GCn)$ is now defined by
$$
  \langle  f, g\rangle_{\wt, \Gn }^{r,\,s} =
   \langle \eta_n (\wt ) * f, g \rangle_{\Gn}^{r,\,s} ,
   \quad f, g \in R( \GCn).
$$
We shall see in Corollary~\ref{cor_char} that $\eta_n (\wt)$ is
self-dual if the class function $\wt$ is invariant under the
antipode $S$. In such a case the matrix of the bilinear form
$\langle \ , \ \rangle_{\wt}^{r,\,s}$ is equal to its adjoint
(transpose and bar action).

We can naturally extend   $\eta_n$ to a map from $R(\G)\otimes r^m
s^n$ to $\RGC$ as follows. In particular, if $\beta$ and $\g $ are
characters of representations $V$ and $W$ of $\G$ respectively, then
\begin{align}  \nonumber
&\eta_n (\beta\otimes r^ms^n +\g\otimes r^ks^l) \\      \label{eq_virt}
  =\sum_{m =0}^n &Ind_{\G_{n -m}{\times}\Cs{\times}\Cs{\times}\Gm{\times}\Cs{\times}\Cs }^{\GCn}
   [ \eta_{n -m} (\beta\otimes r^ms^n) \otimes \eta_m (\g\otimes r^ks^l) ],\\
  &\eta_n (\beta\otimes r^ms^n - \g\otimes r^ks^l) \nonumber\\
  =\sum_{m =0}^n &( -1)^m Ind_{\G_{n -m}{\times}\Cs{\times}\Cs{\times}\Gm{\times}\Cs{\times}\Cs }^{\GCn}
   [ \eta_{n -m} (\beta\otimes  r^ms^n) \otimes \varepsilon_m (\g\otimes r^ks^l) ] .  \label{eq_virt'}
\end{align}

On $\RGC = \bigoplus_{n} R(\GCn)$ the weighted
bilinear form is given by
\[
\langle u, v \rangle_{\wt}^{r,\, s}
 = \sum_{ n \geq 0} \langle u_n, v_n \rangle_{\wt, \Gn }^{r\, s}
\]
where
$u = \sum_n u_n$ and $v = \sum_n v_n$ with $u_n, v_n\in R(\Gn{\times}\Cs{\times}\Cs)$.

 The bilinear form $\langle \ , \ \rangle_{\wt}^{r,\, s}$ on
$\RGC$
is $\mathbb C$-bilinear and takes values in $\mathbb C[r^{\pm1}, s^{\pm1}]$.
When $n =1$, it reduces
 to the weighted bilinear form defined on $\RGC$.

We will often omit the superscript $r\, s$ and use the notation $\langle \ , \ \rangle_{\xi}$
for the weighted bilinear form on $\RGC$.

\section{Two-parameter quantum McKay weights}\label{sect_mckay}

\subsection{Two-parameter quantum McKay correspondence}
Let $d_i = \g_i (c^0)$ be the dimension of
the irreducible representation of $\G$
corresponding to the character $\g_i$.

The following generalizes a result of McKay \cite{Mc}.

\begin{prop} \label{P:qMc}
    For each class $c\in \Gamma_*$ the column vector
 $$
   v(c)= ( \g_0 (c), \g_1 (c), \ldots, \g_{n-1} (c) )^t
   $$
 is an eigenvector of the $n\times n$-matrix
 $A^{r\, s}=(\langle\g_i, \g_j\rangle_{\xi}^{r,\, s})$
 with eigenvalue $ \wt^{r,\, s} (c)$.
 In particular $(d_0, d_1, \ldots, d_{n-1})$ is an eigenvector
 of $A^{r ,\, s}$ with eigenvalue $\wt^{r,\, s}(c^0)$.
\end{prop}
\begin{proof} We compute directly that
\begin{align*}
\sum_{k=0}^r\langle\g_i, \g_k\rangle_{\xi}^{r,\,
s}\g_k(c)&=\sum_k\sum_{c'\in\Gamma_*}
\zeta_{c'}^{-1}\xi^{r,\, s}(c')\g_i(c')\g_k({c'}^{-1})\g_k(c)\\
&=\sum_{c'\in\Gamma_*}\zeta_{c'}^{-1}\xi^{r,\, s}(c')\g_i(c')\sum_k
\g_k({c'}^{-1})\g_k(c)\\
&=\sum_{c'\in\Gamma_*}\zeta_{c'}^{-1}\xi^{r,\, s}(c')\g_i(c')\zeta_c\delta_{cc'}\\
&=\xi^{r,\, s}(c)\g_i(c).
\end{align*}
\end{proof}

  Let $\pi$ be an irreducible faithful
representation $\pi$ of $\G$ of
dimension $d$. For each integer $n$ we
define the $r, s$-integer $[n]$ that can be viewed as a
character of $\Cs{\times}\Cs$ by
\begin{equation}\label{E:qnum}
[n]=\frac{r^n-s^{n}}{r-s}=r^{n-1}+r^{n-2}s+\cdots +rs^{n-2}+s^{n-1}.
\end{equation}
We take the following
special class function
\begin{equation}\label{E:weight}
  \wt=\g_0\otimes [d](rs)^{-\frac{d}{4}} - \pi\otimes 1_{\Cs{\times}\Cs},
\end{equation}
where we have also used the symbol $\pi$ for the corresponding
character, and $1_{\Cs{\times}\Cs}=r^0s^0$ is the trivial character
of $\Cs{\times}\Cs$.

Similar to one-parameter quantum case, we have the following
fact(cf. \cite{FJW2}).
\begin{prop} \label{P:nondeg}
The weighted bilinear form associated to
(\ref{E:weight}) is
non-degenerate. If $\pi$ is an embedding of $\G$
into $SU_d$ and
$t\neq 1$ is a nonnegative real number,
then the weighted bilinear form evaluated on $t$ is positive definite.
\end{prop}

\begin{remark} The matrix $A^{1,\,1}$ is integral, and
the entries of $A^{r,\, s}$ are the $r,\, s$-numbers of the
corresponding entries in $A^{1,\, 1}$ when $r\geq 2$.
\end{remark}

\subsection{Two quantum McKay weights} \label{S:Mcweights}
Let $\G$ is a finite subgroup of $SU_2$ and we introduce
the first distinguished self-dual class function
$$\xi=\g_0\otimes ((rs^{-1})^{\frac{1}{2}}+(r^{-1}s)^{\frac{1}{2}})-\pi\otimes 1_{\Cs{\times}\Cs},
$$
where $\pi$ is the character of the embedding of $\G$ in $SU_2$.

The matrix of the
weighted bilinear form $\langle \ , \ \rangle_{\xi}$ (cf. (\ref{E:qcartan}))
has the following entries:
\begin{equation} \label{E:qcartan1}
a_{ij}=\begin{cases} (rs^{-1})^{\frac{1}{2}}+(r^{-1}s)^{\frac{1}{2}}, & \mbox{if } i=j,\\
-1, & \mbox{if $\langle \g_i, \g_j\rangle_{\xi}^1=-1$,}
\end{cases}
\end{equation}

In particular when $r=s=1$ the matrix $(a_{ij}^{1,\,1})$ coincides
with the extended Cartan matrix of ADE type according to the five
classes of finite subgroups of $SU_2$: the cyclic, binary dihedral,
tetrahedral, octahedral, and icosahedral groups. McKay \cite{Mc}
gave a direct correspondence between a finite subgroup of $SU_2$ and
the affine Dynkin diagram $D$ of ADE type. Each irreducible
character $\g_i$ corresponds to  a vertex of $D$, and the number of
edges between $\g_i$ and $\g_j$ ($i\neq j)$ is equal to $|\langle
\g_i, \g_j\rangle_{\xi}^{1,\, 1}|$, where $\langle \g_i,
\g_j\rangle_{\xi}^{1,\, 1}=a_{ij}^1$ are the entries of matrix
$A^{1\, 1}$ of the weighted bilinear form $\langle \ , \
\rangle_{\xi}^{1,\, 1}$. For this reason we will call our matrix
$A^{r,\, s}=(a_{ij})=(\langle\g_i, \g_j\rangle_{\xi}^{r,\,s})$ the
quantum Cartan matrix. 

\section{Two-parameter quantum Heisenberg algebras and $\Gn$} \label{sect_heis}
\subsection{Two-parameter Heisenberg algebra $\hg $}
 Let $\hg $ be
the infinite dimensional Heisenberg algebra over $\mathbb C[r^{\pm
1}, s^{\pm1}]$, associated with $\Gamma$ and $\wt\in \RGC$, with
generators $a_m(c), c\in\G_*, m\in \mathbb Z$ and a central element
$C$ subject to the following commutation relations:
\begin{equation}  \label{eq_heis}
[a_m(c^{-1}), a_n(c')]
 = m \delta_{m, -n} \delta_{c, c'}\zeta_{c}\xi_{r^m,\,s^{m}}(c)C,
 \quad c, c ' \in \G_*.
\end{equation}

For $m\in \mathbb Z, \g\in \G^*$ and $k,\, l\in \mathbb Z$ we define
\begin{equation*}
  a_m( \g\otimes r^k s^l )
    = \sum_{c \in \G_*} \zeta_c^{ -1}
  {\g } (c) a_m(c)r^{mk}s^{ml}
\end{equation*}
and then extend it to $\RGC$ linearly over $\mathbb C$. Thus
we have for $\g \in \RGC$
\begin{equation}\label{eq_real}
  a_m(\g)
    = \sum_{c \in \G_*} \zeta_c^{ -1}
  {\g_{r^m,\,s^{m}} } (c) a_m(c).
\end{equation}
In particular we have $a_m(\g\otimes r^ks^l)=a_m(\g)r^{mk}s^{ml}$.

It follows immediately from the orthogonality (\ref{eq_orth}) of the irreducible
characters of $\G $ that for each $c\in \G_*$
$$
 a_{ m}( c) = \sum_{ \g\in \G^*} S(\gamma(c)) a_m( \g ).
$$
Note that this formula is also valid if the summation runs through
$\G^*\otimes r^ks^l$ with a fixed $k$ and $l$.

\begin{prop}  \label{prop_orth} The Heisenberg
 algebra $\hg$ has a new basis given by
$a_{n}(\g )$ and $C$ ($ n\in \mathbb Z, \g \in \G^*$) over $\mathbb
C[r^{\pm1}, s^{\pm1}]$
 with the following relations:
 \begin{equation} \label{E:heisen}
  [ a_m(\g), {a_n( \g' )}]=m\delta_{m, -n}\langle\g, \g'
\rangle_{\xi}^{r^m ,\, s^{m}}C.
 \end{equation}
\end{prop}

\begin{proof} This is proved by a direct computation using
Eqns. (\ref{eq_heis}),
 (\ref{eq_twist}) and (\ref{eq_orth}).
 \begin{eqnarray*}
  [ a_m(\g), a_n( \g')]
  & =& \sum_{c, c'\in \G_*} \zeta_c^{-1}\zeta_{c'}^{-1}
{ \g  } (c)\g'(c')
       [ a_m( c), a_n ({c'})]          \\
& =& m \delta_{m, -n}
         \sum_{c, c'\in \G_*}\zeta_{c}^{-1}\zeta_{c'}^{-1}
{ \g  } (c)\g'(c') \delta_{c^{-1}, c'}\zeta_c\xi_{r^m,\,s^{m}}(c)C    \\
  & =& m \delta_{m, -n}
         \sum_{c\in \G_*}\zeta_{c}^{-1}
{ \g  } (c)\g'(c^{-1}) \xi_{r^m,\,s^{m}}(c)C   \\
  & =& m \delta_{m, -n} \langle\g, \g'\rangle_{\xi}^{r^m,\,s^{m}}C.
\end{eqnarray*}
\end{proof}

\subsection{Action of $\hg$ on the Space $\SGC$}
 Let $\SGC $ be the symmetric algebra
generated by $a_{-n}(\g), n \in \mathbb N,
\g\in \Gamma_*$ over $\mathbb C[r^{\pm 1}, s^{\pm1}]$.
We define $a_{-n}(\g\otimes r^ks^l)=a_{-n}(\g)r^{-kn}s^{-ln}$
and the natural degree operator on the space $\SGC $ by
$$
  \deg (a_{ -n}( \g\otimes r^ks^l)) = n
$$
which makes $\SGC $ into a $\mathbb Z_+$-graded algebra.

The space $\SGC$ affords a natural realization of the
Heisenberg algebra $\hg$ with $C=1$.
Since $a_{-n}(\g\otimes r^ks^l)=r^{-nk}s^{-nl}a_{-n}(\g)$, it is enough to describe
the action for $a_{-n}(\g)$.
The central element $C$ acts as the identity operator.
For $n>0$, $a_{-n}( \g)$ act as
multiplication operators on $ \SGC $.
The element $a_n (\g), n \geq 0$ acts as a differential operator through
contraction:
\begin{eqnarray*}
 & a_n (\g). a_{-n_1}( \alpha_1) a_{-n_2} (\alpha_2)
    \ldots a_{-n_k}( \alpha_k)   \\
 &= \sum_{i =1}^k
 \langle \g , \alpha_i \rangle_{\wt}^{r^n,\,s^n}
   a_{-n_1}( \alpha_1) a_{-n_2}(\alpha_2) \ldots
  \check{a}_{-n_i}( \alpha_i) \ldots a_{-n_k}(\alpha_k )  .
\end{eqnarray*}
Here $n_i > 0, \alpha_i \in R(\G)$ for $i =1, \ldots , k$,
and $\check{a}_{-n_i}( \alpha_i)$ means the very term
is deleted.
In this case  $\SGC$ is an irreducible
representation of $\hg$ with the unit $1$ as the highest weight vector.

\subsection{The bilinear form on $\SGC $}
       As a $\hg$-module, the space $ \SGC $ admits a bilinear form
$\langle \ ,  \ \rangle_{\wt } '$ over $\mathbb C[r^{\pm 1}, s^{\pm1}]$ characterized by
\begin{align} \nonumber
\langle 1, 1\rangle_{\xi}'&=1,\\   \label{eq_form}
\langle au, v\rangle_{\xi}'&=\langle u, a^*v\rangle_{\xi}', \qquad a\in \hg,
\end{align}
with the adjoint map $*$ on $\hg$ given by
\begin{equation}  \label{eq_hermit}
a_n(\g\otimes r^ks^l)^* = a_{-n}(\g\otimes r^{k}s^l), \qquad n\in
\mathbb Z.
\end{equation}
Note that the adjoint map $*$ is a $\mathbb C$-linear
anti-homomorphism of $\hg$, and $r^*=\overline r, s^*=\overline s$.
We still use the same symbol $*$ to denote the hermitian-like dual,
since it clearly generalizes the $*$-action on the deformed Cartan
matrix (\ref{E:2varCartan}).

  For any partition $\lambda=( \lambda_1, \lambda_2, \dots)$ and $\g \in \G^*$,
we define
$$
  a_{-\lambda}( \g) = a_{-\lambda_1}( \g)a_{ - \lambda_2}( \g) \dots .
$$
For $\rho = ( \rho (\g) )_{ \g \in \G^*} \in {\mathcal P}(\G^* )$,
we define
$$
  a_{ - \rho\otimes r^ks^l } = r^{-k\|\rho\|}s^{-l\|\rho\|}\prod_{\g \in \G^*}  a_{ - \rho (\g)}(\g).
$$
It is clear that for a fixed $k\in \mathbb Z$ the
elements $a_{ - \rho\otimes q^k},
\rho \in {\mathcal P}(\G^* )$ form a basis of
$\SGC $ over $\mathbb C[r^{\pm 1}, s^{\pm1}]$.

Given a partition $ \lambda = ( \lambda_1, \lambda_2, \ldots )$
and $c \in \G_*$, we define
\begin{eqnarray*}
  a_{ - \lambda } (c\otimes r^ks^l ) & =& r^{-k|\lambda|}s^{-l|\lambda|}
a_{ - \lambda_1}(c) a_{ - \lambda_2 } (c) \ldots, \\
\end{eqnarray*}
For any $\rho = ( \rho (c) )_{ c \in \G_* } \in
 \mathcal P ( \G_* )$ and $k\in \mathbb Z$, we define
\begin{equation*}
  a_{- \rho\otimes r^ks^l}' = r^{-k\|\rho\|}s^{-l\|\rho\|}
\prod_{ c \in \G_*} a_{ - \rho (c)} (c).
\end{equation*}

It follows from
Proposition~\ref{prop_orth} that
\begin{eqnarray*}  \label{eq_inner}
  \langle a_{ - \rho\otimes r^ms^n}', {a_{ - \overline{\sigma}\otimes r^{k}s^l}'}
\rangle_{\wt }'
  = \delta_{\rho, \sigma}r^{\|\rho\|(k-m)}s^{\|\rho\|(l-n)}
    Z_{\rho} \prod_{c \in \G_*} \prod_{i\geq 1}\wt_{q^i}(c)^{m_i(\rho (c))},
 \end{eqnarray*}
where $\rho, \sigma \in \mathcal P (\G_*)$. Note that
$S(a_{-\rho\otimes r^ks^l}')=a_{-\overline{\rho}\otimes
r^{-k}s^{-l}}'$, where we recall that $\overline{\rho} \in {\mathcal
P}(\G_* )$ is the partition-valued function given by $c\mapsto
\rho(c^{-1})$, $c\in \G$.

\section{The characteristic map as an isometry}
\label{sect_isom}
\subsection{The characteristic map $\ch$}
  Let $\Psi : \Gn \rightarrow \SGC$ be the map defined
by $\Psi (x) = a_{ - \rho}'$ if $x \in \Gn$ is of type $\rho$.

      We define a $\mathbb C$-linear map
$ch: \RGC \longrightarrow \SGC$ by letting
\begin{align} \nonumber
 ch (f )
 &= \langle f, \Psi \rangle_{\Gn}\\
 &= \sum_{\rho \in \mathcal P(\G_*)} Z_{\rho}^{-1} S(f(\rho))
 a_{ - \rho}',  \label{E:ch}
\end{align}
where $f(\rho)\in \mathbb C[r^{\pm1}, s^{\pm1}]$ is the value of $f$
at the elements of type $\rho$. The map $ch $ is called the {\em
characteristic map}. This generalizes the definition of the
characteristic map in the classical setting (cf. \cite{M, FJW1,
FJW2}).

The space $\SGC$ can also be interpreted as follows. The element
$a_{ -n } (\g ), n >0 , \g \in \G^*$ is identified as the $n$-th
power sum in a sequence of variables $ y_{ g } = ( y_{i\g } )_{i
\geq 1}$. By the commutativity among $a_{-n}(\g)$ ($\g\in\G^*, n>0$)
and dimension counting it is clear that the space $\SGC$ is
isomorphic with the space $\LG$ of symmetric functions indexed by $
\G^*$ tensored with $ \mathbb C[r^{\pm1}, s^{\pm1}]$ (cf. \cite{M}).

Denote by $c_n (c \in \G_*)$ the conjugacy class in $\Gn$ of
elements $(x, q) \in \Gn$ such that $q$ is an $n$-cycle and $x \in
c$. Denote by $\sigma_n (c\otimes r^ks^l )$ the class function on
$\Gn\times\Cs\times\Cs$ which takes values $n
\zeta_ct_1^{-nk}t_2^{-nl}$ (i.e. the order of the centralizer of an
element in the class $c_n$ times $t_1^{-nk}t_2^{-nl}$) on elements
in the class $c_n\times t_1t_2$ and $0$ elsewhere. For $\rho = \{
m_r (c) \}_{r \geq 1, c \in \G_*} \in \mathcal P_n (\G^*)$ and
$k\in\mathbb Z$,
$$\sigma_{\rho\otimes r^ks^l} = r^{nl}s^{nk}
\prod_{a \geq 1, c \in \G_*} \sigma_a (c)^{m_a (c)}
$$
is the class function on $\Gn\times\Cs\times\Cs$ which takes value
$Z_{\rho}t_1^{-nk}t_2^{-nl}$ on the conjugacy class of type
$\rho\times t_1t_2$ and $0$ elsewhere. Given $\g\in \G^*$ and $k, l
\in \mathbb Z$, we denote by $\sigma_n (\g \otimes r^ks^l)$ the
class function on $\GCn$ which takes values $n \g
(c)t_1^{-nk}t_2^{-nl}$ on elements in the class $c_n\times t_1t_2(c
\in \G_*)$ and $0$ elsewhere.

\begin{lemm}  \label{lem_isom}
  The map $ch$ sends $\sigma_{\rho\otimes r^ks^l}$ to $a_{ - \rho\otimes r^ks^l} '$.
 In particular, it sends  $\sigma_n(\g\otimes r^ks^l )$ to $a_{ -n} ( \g\otimes r^ks^l )$ in $\SGC$.
\end{lemm}
\begin{proof} This is verified by the definition
of $ch$ (\ref{E:ch})
and the character values of $\sigma_n$ defined above.
\end{proof}

\begin{prop}
Given $\g\in \G^*$, the
 character value of
$\eta_n(\g\otimes r^ks^l)$ on the conjugacy class $c_{\rho}$ of type
$\rho=(\rho(c))_{c\in\G_*}$ is given by
\begin{equation} \label{E:charval}
\eta_n(\g\otimes r^ks^l)(c_{\rho}) =\prod_{c\in
\G_*}\g(c)^{l(\rho(c))}r^{nk}s^{nl}.
\end{equation}
In particular, we have $\eta_n(\g\otimes
r^ks^l)=\eta_n(\g)r^{nk}s^{nl}$.
\end{prop}
\begin{proof}
We first let $(g, \sigma)$ be an element of $\Gn$ such that $\sigma$ is a cycle of length $n$, say $\sigma=(12\cdots n)$.
 Let $\{e_i\}$ be a basis of $V$, and $\g\otimes r^ks^l$ is afforded
by the action: $(h, t)e_j=\sum_{i}c_{ij}(h)t^ke_i$, where $h\in \G$.
We then have
\begin{align*}
&(g, \sigma, t).(e_{j_1}\otimes e_{j_2}\otimes \cdots\otimes e_{j_n})\\
&=(g_1, t)e_{j_n}\otimes (g_2, t)e_{j_1} \otimes \cdots\otimes (g_n, t)e_{j_{n-1}}\\
&=\sum_{i_1,\ldots, i_n}t^{kn} c_{i_{n}j_n}(g_1)c_{i_1j_1}(g_2)\cdots
c_{i_{n-1}j_{n-1}}(g_n)
e_{i_n}\otimes e_{i_1}\cdots\otimes e_{i_{n-1}}.
\end{align*}

It follows that
\begin{eqnarray*}
  \eta_n (\g\otimes r^ks^l) (c_{\rho}, t) &=& \mbox{trace }(g, \sigma, t)\\
&=& \sum_{j_1, \ldots, j_n}
t^{kn}c_{j_1j_n}(g_1)c_{j_2j_1}(g_2)\cdots
c_{j_nj_{n-1}}(g_n)\\
& =& \mbox{trace } t^{kn}a(g_n) a(g_{n-1}) \ldots a(g_1)   \\
  & =& \mbox{trace } t^{kn}a(g_n g_{n -1} \ldots g_1) = \g (c)r^{kn}s^{ln}(t).
 \end{eqnarray*}

 Given $x\times y \in \Gn$ where
 $x \in \G_r$ and $y \in \G_{n -r}$, by (\ref{E:eta-action}) we clearly have
$$
\eta_n (\g\otimes r^ks^l) (x\times y, t) = \eta_n (\g\otimes r^ks^l)
(x, t) \eta_n (\g\otimes r^ks^l)(y, t).
$$
 This immediately implies the formula.
\end{proof}

A similar argument gives that
\begin{equation}
  \varepsilon_n (\g\otimes r^ks^l ) ( x, t)
    = (-1)^n  \prod_{c\in \G_*}
( - \g (c))^{l(\rho(c))}t^{nk},
   \label{eq_signterm}
 \end{equation}
where $x$ is any element in the conjugacy class of type
$\rho=(\rho(c))_{c\in\G^*}$.

Formula (\ref{E:charval}) is equivalent to the following:
\begin{equation} \label{E:charval''}
\eta_n(\g\otimes r^ks^l)(c_{\rho}, t)=\prod_{c\in \G_*} \prod_{i\geq
1} (\g\otimes r^ks^l)(c, t^{i})^{m_i(\rho(c))}.
\end{equation}

The following result allows us to extend the map from
$\g\in\G^*$ to $R(\G_n)$.

\begin{prop}  \label{prop_exp}
   For any $\g \in R(\G)$, we have
\begin{align}
 \sum\limits_{n \ge 0}  \ch ( \eta_n( \g\otimes r^ks^l ) ) z^n
  &= \exp \Biggl( \sum_{ n \ge 1}
      \frac 1n \, a_{-n}(\g)(r^{-k}s^{-l}z)^n \Biggr), \label{eq_exp} \\
 \sum\limits_{n \ge 0}  \ch ( \varepsilon_n( \g\otimes r^ks^l )  ) z^n
  &= \exp \Biggl( \sum_{ n \ge 1}
      ( -1)^{ n -1} \frac 1n \, a_{-n}(\g)(r^{-k}s^{-l}z)^n \Biggr).
\label{eq_sign}
\end{align}
\end{prop}

\begin{proof} It follows from definition of ch (\ref{E:ch}) and
(\ref{E:charval''}) that

\begin{eqnarray*}
 &&\sum\limits_{n \ge 0}  \ch ( \eta_n( \g\otimes r^ks^l ) ) z^n\\
  &= & \sum_{\rho} Z_{\rho}^{ -1}
         \prod_{c\in \G_*}\prod_{i\geq 1}S(\g_{r^{ik}s^{il}} (c)^{m_i(\rho(c))})
          a_{ -\rho (c) } z^{|| \rho||}r^{-k||\rho||}s^{-l||\rho||}                  \\
 &= & \sum_{\rho} Z_{\rho}^{ -1}
         \prod_{c\in \G_*}\g (c)^{l(\rho(c))}
          a_{ -\rho (c) } (r^{-k}s^{-l}z)^{|| \rho||}                  \\
  &= & \prod_{c\in \G_*} \Bigl ( \sum_{\lambda }
         (\zeta_c^{ -1}\g (c) )^{l (\lambda)}
         z_{\lambda}^{-1} a_{- \lambda} (c) (r^{-k}s^{-l}z)^{|\lambda|} \Bigr )    \\
  &= & \exp \Biggl  ( \sum\limits_{ n \geq 1}
         \frac1n \sum\limits_{c \in \G_*}
           \zeta_c^{ -1} \g(c) a_{-n} (c) (r^{-k}s^{-l}z)^n \Biggl )      \\
  &= & \exp \Biggl( \sum_{ n \ge 1}
         \frac 1n \, a_{-n}(\g )(r^{-k}s^{-l}z)^n \Biggr).
 \end{eqnarray*}
 Similarly we can prove (\ref{eq_sign}) using
 the following identity
\begin{eqnarray}
  \varepsilon_n (\g\otimes r^ks^l ) ( x)
    &= &(-1)^n  \prod_{c\in \G_*} \prod_{i\geq 1}
( - \g_{r^{ik}s^{il}} (c))^{m_i(\rho(c))} \nonumber\\
&= &(-r^ks^l)^n  \prod_{c\in \G_*} \prod_{i\geq 1}
( - \g (c))^{m_i(\rho(c))} \nonumber\\
&=&\varepsilon_n (\g) ( x)r^{nk}s^{nl}.
   \label{eq_signterm'} \nonumber
 \end{eqnarray}

The same argument as in the classical case (cf. \cite{FJW1, FJW2})
by using (\ref{eq_virt}) and (\ref{eq_virt'}) will show that the
proposition holds for linear combination of simple characters such
as $\g\otimes r^ks^l-\beta\otimes q^k$, and thus it is true for any
element $\g\otimes r^ks^l$, where $\g\in  R(\G)$.
\end{proof}

Comparing components we obtain
\begin{eqnarray*}
  \ch (\eta_n (\g\otimes r^ks^l ) )
  &=& \sum\limits_{\lambda}\frac {r^{-kn}s^{-ln}}{z_\lambda}\,
             a_{-\lambda}(\g ), \\
  \ch (\varepsilon_n (\g\otimes r^ks^l ))
  &=& \sum\limits_{\lambda}\frac {r^{-kn}s^{-ln}}{z_\lambda}\,
             ( -1)^{ | \lambda| - l ( \lambda)} a_{-\lambda}(\g ),
\end{eqnarray*}
where the sum runs over all partitions $\lambda$ of $n$.

\begin{coro}  \label{cor_char}
   The formula (\ref{E:charval''}) remains valid when
$\g\otimes r^ks^l$ is replaced by any element $\wt \in R(\GC)$.
 In particular $\eta_n (\wt)$ is self-dual provided that $\wt$ is
invariant under the antipode $S$.
\end{coro}

\subsection{Isometry between $\RGC$ and $\SGC$}
    The symmetric algebra
$\SGC=\SG\otimes \mathbb C[r^{\pm1}, s^{\pm1}]$ has the following
Hopf algebra structure over $\mathbb C$. The multiplication is the
usual one, and the comultiplication is given by
\begin{align*}
\Delta(r^ks^l)&=r^ks^l\otimes r^ks^l\\
  \Delta ( a_n (\g\otimes r^ks^l))
&=a_n(\g\otimes r^ks^l)\otimes r^{nk}s^{nl}+r^{nk}s^{nl}\otimes
a_n(\g\otimes r^ks^l ),
\end{align*}
where $\g\in\G^*$. The last formula is equivalent to the following:
\begin{equation}\label{E:coprod}
\Delta ( a_n (c\otimes r^ks^l )) =a_n(c\otimes r^ks^l)\otimes
r^{nk}s^{nl}+r^{nk}s^{nl}\otimes a_n (c\otimes r^ks^l ),
\end{equation}
where $c\in\G_*$.
The antipode is given by
\begin{align*}
S(r^ks^l )&=r^{l}s^{k},\\
S(a_n (\g\otimes r^ks^l  ))&=-a_n(\g\otimes r^{l}s^{k})
\end{align*}
The antipode commutes with the adjoint (dual) map $*$:
\begin{equation}
*^2=S^2=Id, \qquad S*=*S.
\end{equation}

Recall that we have defined a Hopf algebra structure
on $\RGC$ in Sect.~\ref{sect_wreath}.

\begin{prop} \label{P:isometry1}
  The characteristic map $ \ch: \RGC \longrightarrow \SGC$
 is an isomorphism of Hopf algebras.
\end{prop}

\begin{proof}
It follows immediately from the definition of the comultiplication in
the both Hopf algebras (cf. (\ref{E:comult}) and (\ref{E:coprod})).
\end{proof}

    Recall that we have defined a bilinear form
$\langle \  ,  \, \rangle_{\wt }$ on $\RGC$ and
a bilinear form on $\SGC$ denoted by
$\langle \ , \, \rangle_{\wt }'$, where $\wt$ is a self-dual
class function. The following
lemma is immediate from our definition of
$\langle \ , \, \rangle_{\wt }'$ and the
comultiplication $\Delta$.

\begin{lemm}
  The bilinear form $\langle \  ,  \, \rangle_{\wt } '$ on $\SGC$
 can be characterized by the following two properties:

 1). $\langle a_{ -n} (\beta\otimes r^{k_1}s^{l_1} ), a_{ -m} (\g\otimes r^{k_2}s^{l_2} )
\rangle_{\wt}^{'}
  = \delta_{n, m} r^{n(k_2-k_1)} s^{n(l_2-l_1)}\langle \beta , \g  \rangle_{\wt}^{'} ,
  \quad \beta, \g \in \G^*,$ $ k_1, k_2, l_1, l_2\in\mathbb Z.$

 2). $ \langle f g , h \rangle_{\wt}^{'}
       = \langle f \otimes g, \Delta h
\rangle_{\wt}^{'} ,$
 where $f, g, h \in \SGC $, and the bilinear form on
$\SGC \otimes \SGC$, is induced
 from $\langle \ , \ \rangle_{\wt}^{'}$ on $\SGC$.
\end{lemm}

\begin{theo}  \label{th_isometry}
  The characteristic map is an isometry from the space

 $ (\RGC, \langle \ \ , \ \  \rangle_{\wt })$ to the space
 $ (\SGC, \langle \ \ , \ \  \rangle_{\wt }' )$.
\end{theo}

\begin{proof}
  By Corollary~\ref{cor_char}, the character
 value of $\eta_n (\wt )$ at an element $x$ of type $\rho$ is
 $$
  \eta_n (\wt ) ( x)= \prod_{c\in \G_*}\prod_{i\geq 1} \wt_{q^i} (c)^{m_i(\rho(c))}.
 $$
 Thus it follows from definition that
 \begin{eqnarray*}
   \langle \sigma_{ \rho\otimes r^{k_1}s^{l_1}}, \sigma_{  \rho '\otimes r^{k_2}s^{l_2}}
\rangle_{\wt }
  & =& \sum_{\mu \in \mathcal P_n (\G_*)}
       Z_{\mu}^{ -1} r^{n(k_2-k_1)}s^{n(l_2-l_1)}\wt_{q} (c_{\mu }) \sigma_{\rho} (c_{\mu})
       \sigma_{\rho '} (c_{\mu})                      \\
  & =& \delta_{\rho, \rho '}
       Z_{\rho}^{ -1} r^{n(k_2-k_1)}s^{n(l_2-l_1)}\wt (c_{\rho}) Z_{\rho}Z_{\rho}  \\
  & =& \delta_{\rho, \rho '}
       Z_{\rho}  r^{n(k_2-k_1)}s^{n(l_2-l_1)}
\prod_{c \in \G_*}\prod_{i\geq 1} \wt_{q^i} (c)^{m_i(\rho(c))}.
 \end{eqnarray*}

  By Lemma~\ref{lem_isom} and the formula (\ref{eq_inner}), we
 see that
 \[
  \langle \sigma_{ \rho\otimes r^{k_1}s^{l_1}}, \sigma_{  \rho '\otimes r^{k_2}s^{l_2}}
\rangle_{\wt }
  = \langle a_{- \rho\otimes r^{k_1}s^{l_1}}, a_{ - \rho '\otimes r^{k_2}s^{l_2}} \rangle_{\wt } '
  =  \langle \ch (\sigma_{ \rho\otimes r^{k_1}s^{l_1}}),
             \ch (\sigma_{  \rho '\otimes r^{k_2}s^{l_2}} ) \rangle_{\wt } '.
 \]
 Since $\sigma_{\rho\otimes r^{k}s^{l}}, \rho \in \mathcal P(\G_*)$ form
 a $\mathbb C$-basis of $\RGC$, we have shown that
 $\ch: \RGC \longrightarrow \SGC$ is an isometry.
\end{proof}


From now on we will not distinguish the bilinear form $\langle \ , \
\rangle_{\wt}$ on $\RGC$
 from the bilinear form $\langle \ , \ \rangle_{\wt}^{'}$ on $\SGC$.

\section{Two-parameter quantum vertex operators and $\RGC$}
\label{sect_vertex}
\subsection{Two-parameter vertex operators and Heisenberg algebras in
$\FGC$}
 Let $\mathbb{K}=\mathbb{Q}(r,s)$ denote a
field of rational functions with two-parameters $r$, $s$ ($r\ne \pm
s$). Let $Q$ be an integral lattice with basis $\a_i$, $i=0, 1,
\ldots, n-1$ endowed with a symmetric bilinear form. We fix a
$2$-cocycle $\epsilon: Q\times Q \longrightarrow
\mathbb{K}^{\times}$ has the following properties:
\begin{eqnarray*}
&&\ep(\a+\beta,\,\gamma)=\ep(\a,\, \beta)\ep(\beta,\, \gamma),\\
&&\ep(\a,\, \beta+\gamma)=\ep(\a,\, \beta)\ep(\a,\, \gamma)
\end{eqnarray*}
We construct such a cocycle directly by
$$\ep(\alpha_i,\alpha_j)=\left\{\begin{array}{cl} (-r_is_i)^{\frac{a_{ij}}{2}},
& \ i> j;\\
(rs)^{\frac{1}{2}},
& \ i= j;\\
1,& \ i<j.\\
\end{array}\right.$$

Let $\xi$ be a self-dual virtual character in $\RGC$. Recall that
the lattice $\Rz$ is a $\mathbb Z$-lattice under  the bilinear form
$\langle \ , \ \rangle_{\wt}^1$, here the superscript means
$r=s^{-1}=1$. For our purpose we will always associate a $2$-cocycle
$\ep$ as in the previous subsection to the integral lattice $(\Rz,
\langle \ , \ \rangle_{\wt}^1 )$ (and its sublattices).

Let $\mathbb C[\Rz]$ be the group algebra
generated by $e^{\g }$, $\g \in \Rz$.
We introduce two special operators acting on $\mathbb C[ \Rz ]$:
A ($\ep$-twisted) multiplication operator $e^{\alpha}$ defined by
 $$
  e^{\alpha }.e^{\beta } = \ep(\alpha, \beta) e^{\alpha +\beta},
 \quad  \alpha, \beta  \in \Rz,
 $$
and a differentiation operator
${\partial_{\alpha }}$ given by
\begin{eqnarray*}
 {\partial_{\a}} e^{ \beta} =
 \langle \a, \beta \rangle_{\wt}^1  e^{ \beta},
 \quad  \alpha, \beta  \in \Rz.
\end{eqnarray*}
These two operators are then extended linearly to the space
\begin{equation}\label{E:fgc}
\FGC = \RGC \otimes \mathbb C[\Rz]
\end{equation}
by letting them act on the $\RGC$ part trivially.

We define the Hopf algebra structure on $\mathbb C[\Rz]$
and extend the Hopf algebra structure from $\RGC$ to $\FGC$as follows.
\begin{equation*}
\Delta(e^{\alpha})=e^{\alpha}\otimes e^{\alpha},
\qquad S(e^{\alpha})=e^{-\alpha}.
\end{equation*}

The bilinear form $\langle\ , \ \rangle_{\wt}^{r\,s}$ on $\RGC$ is
extended to $\FGC$ by
\begin{equation*}
\langle e^{\alpha}, e^{\beta}\rangle_{\wt}=\delta_{\a,\beta}.
\end{equation*}

With respect to
this extended bilinear form we have the $*$-action (adjoint action)
 on the
operators $e^{\alpha}$ and ${\partial}_{\alpha}$:
\begin{equation}
(e^{\alpha})^*=e^{-\alpha}, \qquad
(z^{{\partial}_{\alpha}})^*=z^{-{\partial}_{\alpha}}.
\end{equation}

For each $k\in\mathbb Z$, we introduce the group theoretic operators
\linebreak $ H_{ \pm n}( \g\otimes r^ks^l ), E_{ \pm n} ( \g\otimes
r^ks^l), \g \in R(\G), n > 0 $ as the following compositions of
maps:
\begin{eqnarray*}
  H_{ -n} ( \g\otimes r^ks^l ) &:&
    R ( \GCm )
  \stackrel{ \eta_n (\g\otimes r^ks^l) \otimes}{\longrightarrow}
    R ( \GCn ) \otimes R ( \GCm )\\&&
  \stackrel{ {Ind}\otimes m_{\Cs} }{\longrightarrow}
    R ( \G_{n +m}\times\Cs\times\mathbb C^{\times} )   \\
  E_{ -n} ( \g\otimes r^ks^l ) &:&
    R ( \GCm )
  \stackrel{ \varepsilon_n (\g\otimes r^ks^l) \otimes}{\longrightarrow}
    R ( \GCn ) \otimes R ( \GCm )\\&&
  \stackrel{ {Ind}\otimes m_{\Cs} }{\longrightarrow}
    R ( \G_{n +m}\times\Cs\times\mathbb C^{\times} )   \\
  E_n ( \g\otimes r^ks^l ) &:&
    R ( \GCm )
   \stackrel{ {Res} }{\longrightarrow}
    R( \Gn) \otimes R( \G_{m -n}\times\Cs\times\mathbb C^{\times})\\&&
   \stackrel{ \langle \varepsilon_n (\g\otimes r^ks^l),
         \cdot \rangle_{\wt } }{\longrightarrow}
    R ( \G_{m -n}\times\Cs\times\mathbb C^{\times}) \\
  H_n(\g\otimes r^ks^l ) &:&
    R ( \GCm )
   \stackrel{ {Res} }{\longrightarrow}
    R( \Gn) \otimes R( \G_{m -n}\times\Cs\times\mathbb C^{\times})\\&&
   \stackrel{ \langle \eta_n (\g\otimes q^k),
\cdot \rangle_{\wt} }{\longrightarrow}
    R ( \G_{m -n}\times\Cs\times\mathbb C^{\times}) ,
\end{eqnarray*}
where $Res$ and $Ind$ are the restriction and induction functors
in $R_{\G}=\bigoplus_{n\geq 0}R(\G_n)$.

\medskip

We introduce their generating functions in a formal variable $z$:
\begin{eqnarray*}
  H_{\pm} (\g\otimes r^ks^l, z) &=& \sum_{ n\geq 0} H_{ \mp n}
( \g\otimes r^ks^l ) z^{\pm n}, \\
  E_{\pm} (\g\otimes r^ks^l, z) &=& \sum_{ n\geq 0} E_{ \mp n}
( \g\otimes r^ks^l )( -z)^{\pm n}.
\end{eqnarray*}
  We now define
the vertex operators $Y_n ^{\pm}(\g\otimes r^ks^l, a, b)$ , $\g\in
\G^*$, $k, l, a, b\in \mathbb Z$, $n \in {\mathbb Z} + \langle \g,
\g \rangle_{ \wt}^1 /2$ as follows.
\begin{align}  \nonumber
 Y ^{+}( \g\otimes r^ks^l, a, b, z)
  & = \sum\limits_{n \in
  {\mathbb Z} + \langle \g, \g \rangle_{ \wt}^1 /2}
 Y_n^{+}( \gamma\otimes r^ks^l, a, b)
z^{ -n - \langle \g, \g \rangle_{ \wt}^1 /2}
\\
  & =  H_+ (\g\otimes r^ks^l, z) E_- (\g\otimes r^{k-a}s^{l-b} , z) e^{ \g}
(r^{-k}s^{-l}z)^{ \partial_{ \g}}, \label{eq_vo}
\end{align}
\begin{align}\nonumber
Y^-(\g\otimes r^ks^l, a, b, z)&=(Y^+(\g\otimes r^ks^l, a, b, z^{-1}))^*\\
\nonumber &=\sum\limits_{n \in
  {\mathbb Z} + \langle \g, \g \rangle_{ \wt}^1 /2}
 Y_n^{-}( \gamma\otimes r^ks^l, a, b)
z^{ -n - \langle \g, \g \rangle_{ \wt}^1 /2}\\
&=E_+ (\g\otimes r^{k-a}s^{l-b}, z) H_- (\g\otimes r^ks^{l} , z) e^{
-\g} (r^{-k}s^{-l}z)^{ -\partial_{ \g}}.
\end{align}

One easily sees that the operators $Y_n^{\pm} (\g\otimes r^ks^{l},
a, b)$ are well-defined operators acting on the space $\FGC$.

We extend the $\mathbb Z_+$-gradation
on $\RGC$ to a $\frac12\langle \g , \g \rangle_{\wt}^1 +
 \mathbb Z_+$-gradation on
$\FGC$ by letting
\begin{eqnarray*}
  \deg  a_{ -n} (\g\otimes r^ks^l ) = n , \quad
  \deg e^{\g } = \frac12 \langle \g , \g \rangle_{\wt}^1 .
\end{eqnarray*}

We denote by $\RGGC$ the subalgebra of $\RGC$
excluding the generators
$a_n(\g_0)$, $n\in \mathbb Z^{\times}$.
The bilinear form $\langle \ , \ \rangle_{\xi}$ on
$$\FGGC=\RGGC\otimes\Rzz
$$
will be the restriction of $\langle \ , \ \rangle_{\xi}$
on $\FGC$ to $\FGGC$.

%
%
%
%

We define $ \widetilde{a}_{ -n} (\gamma\otimes r^ks^l), n >0$ to be
a map from $\RGC$ to itself by the following composition
\begin{align*}
  R (\GCm) \stackrel{ \sigma_n ( \g\otimes r^ks^l ) \otimes }{\longrightarrow}&
  R(\GCn) \otimes R (\GCm) \\
& \stackrel{{Ind\otimes m_{\Cs}} }{\longrightarrow}
  R ( {\Gamma}_{n +m}\times \Cs\times\mathbb C^{\times}).
\end{align*}
We also define $ \widetilde{a}_{ n} (\gamma\otimes r^ks^l), n >0$ to
be a map from $\RGC$ to itself as the composition
\begin{align*}
  R (\GCm)  \stackrel{ Res \otimes 1}{\longrightarrow}&
   R(\GCn)\otimes R ( {\G }_{m -n}\times \Cs\times\mathbb C^{\times})\\
 &\stackrel{ \langle \sigma_n ( \g\otimes r^ks^l), \cdot \rangle_{\wt}^q}{\longrightarrow}
 R ( {\G }_{m -n}\times\Cs\times\mathbb C^{\times}).
\end{align*}

\begin{prop} The operators $\widetilde{a}_{n} (\gamma)$,
$\g\in \G^*, n\in\mathbb Z^{\times}$ satisfy the
Heisenberg algebra relations (\ref{eq_heis}) with $C=1$.
\end{prop}
\begin{proof} This is similarly proved as for the classical
setting in \cite{W}.
\end{proof}

\subsection{Group theoretic interpretation of vertex operators}
To compare the vertex operators $Y^{\pm}(\g\otimes q^k, a, b, z)$
with the familiar vertex operators acting in the Fock space we
introduce the space
$$
  \VGC = \SGC \otimes \mathbb C [ \Rz].
$$

We extend the bilinear form $\langle \  , \ \rangle_{\wt }^{r,\, s}$
in $\SGC$ to the space $\VGC$ and also extend the $\mathbb
Z_+$-gradation on $\SGC$ to a $\frac12 \mathbb Z_+$-gradation on
$\VG$.

We extend the characteristic map to the map
\begin{equation*}
ch: \FGC\longrightarrow \VGC
\end{equation*}
by identity on $\Rz$. Then Proposition \ref{P:isometry1}
and Theorem \ref{th_isometry}
imply that we have an isometric isomorphism of Hopf algebras.
We can now identify the operators from the previous
subsections with the operators constructed
from the Heisenberg algebra.

\begin{theo}
\label{T:characteristic}
For any $\g \in R(\G)$ and $k\in\mathbb Z$,  we have
  \begin{eqnarray} \label{E:ch1}
   \ch \bigl ( H_+ (\g\otimes r^ks^l, z) \bigl )
   &=& \exp \biggl ( \sum\limits_{ n \ge 1} \frac 1n \,
    a_{-n} ( \g ) (r^{-k}s^{-l}z)^n \biggr ), \\ \label{E:ch2}
   \ch \bigl ( E_+ (\g\otimes r^ks^l, z) \bigl )
   &=& \exp\biggl ( -\sum\limits_{n\ge 1}\frac 1n  \,
    a_{-n}(\gamma)(r^{-k}s^{-l}z)^n\biggr ),  \\ \label{E:ch3}
   \ch \bigl ( H_- (\g\otimes r^ks^l , z) \bigl )
   &=& \exp \biggl ( \sum\limits_{n \ge 1}\frac 1n \,
    a_n(\g) (r^{-k}s^{-l}z)^{-n}\biggr ),         \\ \label{E:ch4}
   \ch \bigl ( E_- (\g\otimes r^ks^l , z) \bigl )
   &=& \exp\,\, \biggl ( -\sum\limits_{ n \ge 1}
    \frac 1n \,{ a_n( \g )} (r^{-k}s^{-l}z)^{ -n} \biggr ) .
  \end{eqnarray}
\end{theo}

\begin{proof}
  The first and second identities were essentially
 established in Proposition~\ref{prop_exp} together with
 Lemma~\ref{lem_isom}, where the components are viewed as
operators acting on $\RGC$ or $\SGC$. Note that $a_n(\g\otimes
r^ks^l)=a_n(\g)r^{kn}s^{ln}$.

 We observe from definition that
 the adjoint $*$-action
of $E_+ (\g\otimes r^ks^l, z)$ and $H_- (\g\otimes r^ks^l , z)$
 with respect to the bilinear form
 $\langle \ , \ \rangle_{\wt}^{r,\, s}$
 are $E_- (\g\otimes r^{k}s^{l} , z^{-1})$ and $ H_- (\g\otimes r^{k}s^{l} , z^{-1})$
 respectively.
 The third and fourth identities are obtained by applying
 the adjoint action $*$ to the first two identities.
\end{proof}

\begin{remark} Replacing $\g$ by $-\g$ in (\ref{E:ch1}) and (\ref{E:ch3})
we obtain the equivalent formulas (\ref{E:ch2}) and (\ref{E:ch4})
respectively.
\end{remark}

Applying the characteristic map to the vertex operators $Y^{\pm}(\g,
a, b, z)$, we obtain the following group theoretical explanation of
vertex operators acting on the Fock space $\FGC$.

\begin{theo} \label{T:vertexop}
For any $\g\in\RG$ and $k\in\mathbb Z$, we have
\begin{align*}
  &Y^{+}( \g , a, b, z)\\
  &= \exp \biggl ( \sum\limits_{ n \ge 1}
  \frac 1n \, \widetilde{a}_{-n} ( \g ) z^n \biggr ) \,
  \exp \biggl ( -\sum\limits_{ n \ge 1}
  \frac 1n \,{ \widetilde{a}_n( \g)} r^{-an}s^{-bn} z^{ -n} \biggr )
  e^{ \g} z^{ \partial_{\g }}\\
&=ch(H_+(\g, z))ch(S(H_+(\g\otimes r^{a}s^{b}, z^{-1})^*)) e^{ \g}
z^{
\partial_{\g }},
\end{align*}

 \begin{align*}
  &Y^{-}( \g , a, b, z)\\
  &= \exp \biggl ( -\sum\limits_{ n \ge 1}
  \frac 1n \, \widetilde{a}_{-n} ( \g )r^{an}s^{bn} z^n \biggr ) \,
  \exp \biggl ( \sum\limits_{ n \ge 1}
  \frac 1n \,{ \widetilde{a}_n( \g)}  z^{ -n} \biggr )
  e^{ -\g} z^{-\partial_{\g }}\\
&=ch(S(H_+(\g\otimes r^{a}s^{b}, z^{-1}))) ch(H_+(\g, z)^*)e^{-\g}
z^{-\partial_{\g }}.
\end{align*}
\end{theo}

We note that for $\g\in\G^*, l\in\mathbb Z$
\begin{equation}
Y^{\pm}(\g\otimes r^ks^l, a, b, z)=Y^{\pm}(\g, a, b, r^{-k}s^{-l}z).
\end{equation}

It follows from Theorem \ref{T:vertexop} that
\begin{align*}
&ch\big(Y^{+}(\g, a, b, z)\big)=X^{+}(\g, a, b, z)\\
&= \exp \biggl ( \sum\limits_{ n \ge 1}
  \frac 1n \, {a}_{-n} ( \g ) z^n \biggr )\\
 &\qquad\times \exp \biggl (-\sum\limits_{ n \ge 1}
  \frac 1n \,{ {a}_n( \g)}r^{-an}s^{-bn}  z^{ -n} \biggr )
  e^{ \g} z^{\partial_{\g }}.
\end{align*}
\begin{align*}
&ch\big(Y^{-}(\g, a, b, z)\big)=X^{-}(\g, a, b, z)\\
&= \exp \biggl ( -\sum\limits_{ n \ge 1}
  \frac 1n \, {a}_{-n} ( \g )r^{an}s^{bn} z^n \biggr )\\
 &\qquad\times \exp \biggl ( \sum\limits_{ n \ge 1}
  \frac 1n \,{ {a}_n( \g)} z^{ -n} \biggr )
  e^{ -\g} z^{-\partial_{\g }}.
\end{align*}

When $r=s^{-1}=q$ they specialize to the vertex operators $Y^{\pm}(
\g ,k, z)$
 studied in \cite{FJW2}.

Under the new variable (by identifying $a_i(n)$ with
$\widetilde{a_i}(n)$) we obtain that
\begin{align*}
&X^{+}(\g_i\otimes s^{-b}, a, b, z)\\
&=\exp \biggl ( \sum\limits_{ n \ge 1}
  \frac {a_{i} (-n )}{[n]} s^{-bn}z^n \biggr ) \,
  \exp \biggl ( -\sum\limits_{ n \ge 1}
  \frac {a_i(n)}{[n]} r^{-an} z^{ -n} \biggr )
  e^{ \g} z^{ \partial_{\g }} ,\\
&X^{-}(\g_i\otimes r^{-a}, k,  z)\\
&=\exp \biggl ( -\sum\limits_{ n \ge 1}
  \frac {a_{i}(-n )}{[n]}s^{bn} z^n \biggr ) \,
  \exp \biggl ( \sum\limits_{ n \ge 1}
  \frac {a_i(n)}{[n]} r^{an} z^{ -n} \biggr )
  e^{ -\g} z^{ -\partial_{\g }},
\end{align*}
where $[n]$ is the two-parameter quantum number (\ref{E:qnum}).

\section{Basic representations and the McKay
correspondence}
\label{sect_ade}

\subsection{Two-parameter quantum toroidal algebras}
\medskip
In this subsection we define the two-parameter quantum toroidal
algebras$U_{r,s}(\widehat{\widehat{\mathfrak g}})$ of simply laced
type $A, D$ or $E$. In particular the two-parameter quantum toroidal
algebra contains a special subalgebra -the two-parameter quantum
affine algebras  $U_{r,s}(\widehat{\mathfrak g})$(cf. \cite{HRZ}).

Let $(A_{ij})$ be the two-parameter quantum Cartan  $(N+1)\times
(N+1)$- Martix(cf. \cite{HZ, Z}), For type $A_{n}^{(1)}$, we have
$$A_{ij}=\left(\begin{array}{cccccc}
rs^{-1}& r^{-1}& 1 & \cdots & 1 & s \\
s & rs^{-1} & r^{-1}  & \cdots & 1 & 1\\
\cdots &\cdots &\cdots & \cdots & \cdots & \cdots\\
1 & 1 & 1  & \cdots & rs^{-1} & r^{-1}\\
 r^{-1} & 1 & 1 & \cdots & s & rs^{-1}
\end{array}\right)$$

\begin{defi}
The two-parameter quantum toroidal algebra
$U_{r,s}(\widehat{\widehat{\mathfrak g}})$ is an associative algebra
over $\mathbb{K}$ generated by the elements $x_i^{\pm}(k)$,
$a_i(m)$, $\om_i^{\pm1}$, ${\om'_i}^{\pm1}$,
$\gamma^{\pm\frac{1}{2}}$, ${\gamma'}^{\,\pm\frac{1}2}$, $D^{\pm1}$,
$D'^{\,\pm1}$, $(0\le
 i\le N$, $k,\,k' \in \mathbb{Z}$, $m,\,m' \in \mathbb{Z}\backslash
\{0\}$), subject to the following defining relations:

\medskip
\noindent $(\textrm{D1})$ \  $\gamma^{\pm\frac{1}{2}}$,
$\gamma'^{\,\pm\frac{1}{2}}$ are central with $\gamma\gamma'=rs$,
$\omega_i\,\omega_i^{-1}=\omega_j'\,\omega_j'^{\,-1}=1$ $(i, j\in
I)$, and
\begin{equation*}
\begin{split}
[\,\omega_i^{\pm 1},\omega_j^{\,\pm 1}\,]&=[\,\om_i^{\pm1},
D^{\pm1}\,]=[\,\om_j'^{\,\pm1}, D^{\pm1}\,] =[\,\om_i^{\pm1},
D'^{\pm1}\,]=0\\
&=[\,\omega_i^{\pm 1},\omega_j'^{\,\pm 1}\,]=[\,\om_j'^{\,\pm1},
D'^{\pm1}\,]=[D'^{\,\pm1}, D^{\pm1}]=[\,\omega_i'^{\pm
1},\omega_j'^{\,\pm 1}\,].
\end{split}
\end{equation*}

$$[\,a_i(m),a_j(m')\,]=\delta_{m+m',0}\frac{(rs)^{\frac{|m|}2}
(A_{ii}^{\frac{m a_{ij}}2}-A_{ii}^{-\frac{m a_{ij}}2})}{|m|(r-s)}
\cdot\frac{\gamma^{|m|}-\gamma'^{|m|}}{r-s}. \leqno(\textrm{D2})$$
$$[\,a_i(m),~\om_j^{{\pm }1}\,]=[\,\,a_i(m),~{\om'}_j^{\pm
1}\,]=0.\leqno(\textrm{D3})
$$
\begin{gather*}
D\,x_i^{\pm}(k)\,D^{-1}=r^k\, x_i^{\pm}(k), \qquad\ \
D'\,x_i^{\pm}(k)\,D'^{\,-1}=s^k\, x_i^{\pm}(k),\tag{\textrm{D4}}
\\
D\, a_i(m)\,D^{-1}=r^m\,a_i(m), \qquad\quad D'\,
a_i(m)\,D'^{\,-1}=s^m\,a_i(m).
\end{gather*}

$$
\om_i\,x_j^{\pm}(k)\, \om_i^{-1} =  A_{ji}^{\pm 1} x_j^{\pm}(k),
\qquad \om'_i\,x_j^{\pm}(k)\, \om_i'^{\,-1} = A_{ij}
^{\mp1}x_j^{\pm}(k).\leqno(\textrm{D5})
$$

$$
\begin{array}{ll}
&[\,a_i(m),x_j^{\pm}(k)\,]=\pm\frac{(rs)^{\frac{|m|}2}((rs^{-1})^{\frac{m
a_{ij}}2}-(rs^{-1})^{-\frac{m
a_{ij}}2})}{m(r-s)}\hskip3cm\vspace{12pt}\\
&\hskip3.3cm\cdot\gamma^{\pm\frac{m}2}x_j^{\pm}(m{+}k), \qquad \quad
\textit{for} \quad m<0,
\end{array} \leqno{(\textrm{D$6_1$})}
$$

$$
\begin{array}{ll}
&[\,a_i(m),x_j^{\pm}(k)\,]=\pm\frac{(rs)^{\frac{|m|}2}((rs^{-1})^{\frac{m
a_{ij}}2}-(rs^{-1})^{-\frac{m
a_{ij}}2})}{m(r-s)}\hskip3cm\vspace{12pt}\\
&\hskip3.3cm\cdot\gamma'^{\pm\frac{m}2}x_j^{\pm}(m{+}k), \qquad
\quad\textit{for} \quad m>0,
\end{array} \leqno{(\textrm{D$6_2$})}
$$

$$
\begin{array}{lll}
x_i^{\pm}(k{+}1)\,x_j^{\pm}(k') - A_{ji}^{\pm1} x_j^{\pm}(k')\,x_i^{\pm}(k{+}1)\\
=-\Bigl(A_{ji}A_{ij}^{-1}\Bigr)^{\pm\frac1{2}}\,\Bigl(x_j^{\pm}(k'{+}1)\,x_i^{\pm}(k)-A_{ij}^{\pm1}
x_i^{\pm}(k)\,x_j^{\pm}(k'{+}1)\Bigr),
\end{array}\leqno{(\textrm{D7})}
$$

$$
[\,x_i^{+}(k),~x_j^-(k')\,]=\frac{\delta_{ij}}{r-s}\Big(\gamma'^{-k}\,{\gamma}^{-\frac{k+k'}{2}}\,
\om_i(k{+}k')-\gamma^{k'}\,\gamma'^{\frac{k+k'}{2}}\,\om'_i(k{+}k')\Big),\leqno(\textrm{D8})
$$
where $\om_i(m)$, $\om'_i(-m)~(m\in \mathbb{Z}_{\geq 0})$ with
$\om_i(0)=\om_i$ and $\om'_i(0)=\om_i'$ are defined by:
\begin{gather*}
\sum\limits_{m=0}^{\infty}\om_i(m) z^{-m}=\om_i \exp \Big(
(r{-}s)\sum\limits_{\ell=1}^{\infty}
 a_i(\ell)z^{-\ell}\Big); \\
\sum\limits_{m=0}^{\infty}\om'_i(-m) z^{m}=\om'_i \exp
\Big({-}(r{-}s)
\sum\limits_{\ell=1}^{\infty}a_i(-\ell)z^{\ell}\Big),
\end{gather*}
with $\om_i(-m)=0$ and $\om'_i(m)=0, \ \forall\;m>0$.

$$x_i^{\pm}(m)x_j^{\pm}(k)=\langle j,i\rangle^{\pm1}x_j^{\pm}(k)x_i^{\pm}(m),
\qquad\ \hbox{对} \quad a_{ij}=0,\leqno(\textrm{D$9_1$})$$
$$
\begin{array}{lll}
& Sym_{m_1,\cdots
m_{n}}\sum_{k=0}^{n=1-a_{ij}}(-1)^k(r_is_i)^{\pm\frac{k(k-1)}{2}}
\Big[{1-a_{ij}\atop  k}\Big]_{\pm{i}}x_i^{\pm}(m_1)\cdots x_i^{\pm}(m_k) x_j^{\pm}(\ell)\\
&\hskip1.8cm \times x_i^{\pm}(m_{k+1})\cdots x_i^{\pm}(m_{n})=0,
\quad\hbox{对} \quad a_{ij}< 0, \quad  0\leq j<i<N,
\end{array} \leqno{(\textrm{D$9_2$})}
$$
$$
\begin{array}{lll}
& Sym_{m_1,\cdots
m_{n}}\sum_{k=0}^{n=1-a_{ij}}(-1)^k(r_is_i)^{\mp\frac{k(k-1)}{2}}
\Big[{1-a_{ij}\atop  k}\Big]_{\mp{i}}x_i^{\pm}(m_1)\cdots x_i^{\pm}(m_k) x_j^{\pm}(\ell)\\
&\hskip1.8cm \times x_i^{\pm}(m_{k+1})\cdots x_i^{\pm}(m_{n})=0,
\quad\hbox{对} \quad a_{ij}< 0, \quad  0\leq i<j<N,
\end{array} \leqno{(\textrm{D$9_3$})}
$$
where $\textit{Sym}$ denotes symmetrization with respect to the
indices $(m_1, m_2)$,
$$\begin{bmatrix} n\\
k\end{bmatrix}_{\mp}=\frac{[n]_{\mp}!}{[n-k]_{\mp}![k]_{\mp}!},$$
and the $\mp$ specifies the substitution of parameters $r, s$ by
$r^{\mp}, s^{\mp}.$

\end{defi}
The generating functions is defined by:
\begin{gather*}
x_i^{\pm}(z) = \sum_{k \in \mathbb{Z}}x_i^{\pm}(k) z^{-k}, \quad
\om_i(z) = \sum_{m \in \mathbb{Z}_+}\om_i(m) z^{-m}, \quad
  \om_i'(z) = \sum_{n \in -\mathbb{Z}_+}\om_i'(n) z^{-n}. \end{gather*}

\begin{remark} Replacing the index $0\leq i \leq N$ by
$1\leq i \leq N$  in the above definition, we get the corresponding
definition for two-parameter quantum affine algebras (cf. \cite{HZ,
Z}).
\end{remark}

In the case of typr $A$, the two-parameter quantum toroidal algebra
$U_{r,s}(\widehat{\widehat{\mathfrak g}})$ admits a further
deformation $U_{r,s,\kappa}(\widehat{\widehat{\mathfrak g}})$.  Let
$(b_{ij})$ be the skew-symmetric$(N+1)\times(N+1)-$ matrix
$$\left(\begin{array}{cccccc}
0& 1& 0 & \cdots & 0 & -1 \\
-1 & 0 & 1  & \cdots & 0& 0\\
\cdots &\cdots &\cdots & \cdots & \cdots & \cdots\\
0& 0 & 0  & \cdots & 0 & 1\\
 1 & 0 & 0 & \cdots & -1 & 0
\end{array}\right)$$

\begin{defi} Let $\kappa$ be an element of $\mathbb{K}^*$.
The two-parameter quantum toroidal algebra
$U_{r,s,\kappa}(\widehat{\widehat{\mathfrak g}})$ is an associative
algebra over $\mathbb{K}$ generated by the elements $x_i^{\pm}(k)$,
$a_i(m)$, $\om_i^{\pm1}$, ${\om'_i}^{\pm1}$,
$\gamma^{\pm\frac{1}{2}}$, ${\gamma'}^{\,\pm\frac{1}2}$,
$D_1^{\pm1}$, $D_1'^{\,\pm1}$, $D_2^{\pm1}$, $D_2'^{\,\pm1}$ $(0\le
 i\le N$, $k,\,k' \in \mathbb{Z}$, $m,\,m' \in \mathbb{Z}\backslash
\{0\},\, l, l'=1, 2$), subject to the following defining relations:

\medskip
\noindent $(\textrm{T1})$ \  $\gamma^{\pm\frac{1}{2}}$,
$\gamma'^{\,\pm\frac{1}{2}}$ are central with $\gamma\gamma'=rs$,
$\omega_i\,\omega_i^{-1}=\omega_j'\,\omega_j'^{\,-1}=1$ $(i, j\in
I)$, and
\begin{equation*}
\begin{split}
[\,\omega_i^{\pm 1},\omega_j^{\,\pm 1}\,]&=[\,\om_i^{\pm1},
D_l^{\pm1}\,]=[\,\om_j'^{\,\pm1}, D_l^{\pm1}\,] =[\,\om_i^{\pm1},
D_l'^{\pm1}\,]=0\\
&=[\,\omega_i^{\pm 1},\omega_j'^{\,\pm 1}\,]=[\,\om_j'^{\,\pm1},
D_l'^{\pm1}\,]=[D_l'^{\,\pm1}, D_{l'}^{\pm1}]=[\,\omega_i'^{\pm
1},\omega_j'^{\,\pm 1}\,].
\end{split}
\end{equation*}

$$[\,a_i(m),a_j(m')\,]=\delta_{m+m',0}\frac{(rs)^{\frac{|m|}2}
(A_{ii}^{\frac{m a_{ij}}2}-A_{ii}^{-\frac{m a_{ij}}2})}{|m|(r-s)}
\cdot\frac{\gamma^{|m|}-\gamma'^{|m|}}{r-s}\kappa^{mb_{ij}}.
\leqno(\textrm{T2})$$
$$[\,a_i(m),~\om_j^{{\pm }1}\,]=[\,\,a_i(m),~{\om'}_j^{\pm
1}\,]=0.\leqno(\textrm{T3})
$$
\begin{gather*}
D_1\,x_i^{\pm}(k)\,D_1^{-1}=r^k\, x_i^{\pm}(k), \qquad\ \
D_1'\,x_i^{\pm}(k)\,D_1'^{\,-1}=s^k\, x_i^{\pm}(k),\tag{\textrm{T4}}
\\
D_1\, a_i(m)\,D_1^{-1}=r^m\,a_i(m), \qquad\quad D_1'\,
a_i(m)\,D_1'^{\,-1}=s^m\,a_i(m),\\
D_2\,x_i^{\pm}(k)\,D_2^{-1}=r^{\pm \delta_{i0}}\, x_i^{\pm}(k),
\qquad\ \ D_2'\,x_i^{\pm}(k)\,D_2'^{\,-1}=s^{\pm \delta_{i0}}\,
x_i^{\pm}(k),\\
D_2\, a_i(m)\,D_2^{-1}=a_i(m), \qquad\quad D_2'\,
a_i(m)\,D_2'^{\,-1}=a_i(m).
\end{gather*}

$$
\om_i\,x_j^{\pm}(k)\, \om_i^{-1} =  A_{ji}^{\pm 1} x_j^{\pm}(k),
\qquad \om'_i\,x_j^{\pm}(k)\, \om_i'^{\,-1} = A_{ij}
^{\mp1}x_j^{\pm}(k).\leqno(\textrm{T5})
$$
$$
\begin{array}{ll}
&[\,a_i(m),x_j^{\pm}(k)\,]=\pm\frac{(rs)^{\frac{|m|}2}((rs^{-1})^{\frac{m
a_{ij}}2}-(rs^{-1})^{-\frac{m
a_{ij}}2})}{m(r-s)}\hskip3cm\vspace{12pt}\\
&\hskip3.3cm\cdot\gamma^{\pm\frac{m}2}\kappa^{mb_{ij}}x_j^{\pm}(m{+}k),
\qquad \quad \textit{for} \quad m<0,
\end{array} \leqno{(\textrm{T$6_1$})}
$$

$$
\begin{array}{ll}
&[\,a_i(m),x_j^{\pm}(k)\,]=\pm\frac{(rs)^{\frac{|m|}2}((rs^{-1})^{\frac{m
a_{ij}}2}-(rs^{-1})^{-\frac{m
a_{ij}}2})}{m(r-s)}\hskip3cm\vspace{12pt}\\
&\hskip3.3cm\cdot\gamma'^{\pm\frac{m}2}\kappa^{mb_{ij}}x_j^{\pm}(m{+}k),
\qquad \quad\textit{for} \quad m>0,
\end{array} \leqno{(\textrm{T$6_2$})}
$$

\begin{equation*}
\begin{split}
\Big(\kappa^{b_{ij}} z-(A_{ij}A_{ji})^{\pm \frac{1}{2}}w\Big)\,x_i^{\pm}(z)x_j^{\pm}(w)\hskip4cm\\
=\Big(\kappa^{b_{ij}}A_{ji}^{\pm 1} z-(A_{ji}A_{ij}^{-1})^{\pm
\frac{1}{2}}w\Big)\,x_j^{\pm}(w)\,x_i^{\pm}(z).
\end{split}\tag{\textrm{T7}}
\end{equation*}

$$
[\,x_i^{+}(k),~x_j^-(k')\,]=\frac{\delta_{ij}}{r-s}\Big(\gamma'^{-k}\,{\gamma}^{-\frac{k+k'}{2}}\,
\om_i(k{+}k')-\gamma^{k'}\,\gamma'^{\frac{k+k'}{2}}\,\om'_i(k{+}k')\Big),\leqno(\textrm{T8})
$$
where $\om_i(m)$, $\om'_i(-m)~(m\in \mathbb{Z}_{\geq 0})$ with
$\om_i(0)=\om_i$ and $\om'_i(0)=\om_i'$ are defined by:
\begin{gather*}
\sum\limits_{m=0}^{\infty}\om_i(m) z^{-m}=\om_i \exp \Big(
(r{-}s)\sum\limits_{\ell=1}^{\infty}
 a_i(\ell)z^{-\ell}\Big); \\
\sum\limits_{m=0}^{\infty}\om'_i(-m) z^{m}=\om'_i \exp
\Big({-}(r{-}s)
\sum\limits_{\ell=1}^{\infty}a_i(-\ell)z^{\ell}\Big),
\end{gather*}
with $\om_i(-m)=0$ and $\om'_i(m)=0, \ \forall\;m>0$.

$$x_i^{\pm}(m)x_j^{\pm}(k)=\langle j,i\rangle^{\pm1}x_j^{\pm}(k)x_i^{\pm}(m),
\qquad\ \hbox{对} \quad a_{ij}=0,\leqno(\textrm{T$9_1$})$$
$$
\begin{array}{lll}
& Sym_{m_1,\cdots
m_{n}}\sum_{k=0}^{n=1-a_{ij}}(-1)^k(r_is_i)^{\pm\frac{k(k-1)}{2}}
\Big[{1-a_{ij}\atop  k}\Big]_{\pm{i}}x_i^{\pm}(m_1)\cdots x_i^{\pm}(m_k) x_j^{\pm}(\ell)\\
&\hskip1.8cm \times x_i^{\pm}(m_{k+1})\cdots x_i^{\pm}(m_{n})=0,
\quad\hbox{对} \quad a_{ij}< 0, \quad  0\leq j<i<N,
\end{array} \leqno{(\textrm{T$9_2$})}
$$
$$
\begin{array}{lll}
& Sym_{m_1,\cdots
m_{n}}\sum_{k=0}^{n=1-a_{ij}}(-1)^k(r_is_i)^{\mp\frac{k(k-1)}{2}}
\Big[{1-a_{ij}\atop  k}\Big]_{\mp{i}}x_i^{\pm}(m_1)\cdots x_i^{\pm}(m_k) x_j^{\pm}(\ell)\\
&\hskip1.8cm \times x_i^{\pm}(m_{k+1})\cdots x_i^{\pm}(m_{n})=0,
\quad\hbox{对} \quad a_{ij}<0, \quad  0\leq i<j<N.
\end{array} \leqno{(\textrm{T$9_3$})}
$$
$\textit{Sym}$ denotes symmetrization with respect to the indices
$(m_1, m_2)$.
\end{defi}

\begin{remark} Assume that $r=s^{-1}=q$, $U_{r,s,\kappa}(\widehat{\widehat{\mathfrak g}})$
is the one-parameter quantum toroidal algebras $U_{q,
\kappa}(\widehat{\widehat{\mathfrak g}})$ (cf. \cite{GKV, VV}).
\end{remark}

\subsection{A new form of McKay correspondence}

 In this subsection we let $\G$ to be a finite
subgroup of $SU_2$ and consider two distinguished
choices of the class function $\wt$ in $\RGC$ introduced in
Sect. \ref{S:Mcweights}.

First we consider
$$\wt=\g_0\otimes ((rs^{-1})^{\frac{1}{2}}+(r^{-1}s)^{\frac{1}{2}})  - \pi\otimes 1_{\Cs},
$$
where
$\pi$ is the character of the two-dimensional natural representation
of $\G$ in $SU_2$.

The Heisenberg algebra in this case has the following relations (cf.
Prop. \ref{prop_orth} and (\ref{E:qcartan1})).
\begin{equation}\label{E:heisen1}
[a_m(\g_i), a_n(\g_j)]=
\begin{cases}
m\delta_{m, -n}((rs^{-1})^{\frac{m}{2}}+(r^{-1}s)^{\frac{m}{2}})C, & i=j\\
m\delta_{m, -n}a_{ij}^1C, & i\neq j
\end{cases},
\end{equation}
where $a_{ij}^1$ are the entries of the affine Cartan matrix of ADE type
(see (\ref{E:qcartan}) at $d=2$).

Recall that the matrix $A^{1,\,1} = (\langle \g_i,
\g_j\rangle_{\wt}^1) =(a_{ij}^1)_{0 \leq i,j \leq N}$ is the Cartan
matrix for the corresponding affine Lie algebra \cite{Mc}. In
particular $a_{ii}^1 =2$; $a_{ij}^1 =0$ or $-1$ when $i \neq j$ and
$\G \neq \mathbb Z / 2\mathbb Z$. In the case of $\G = \mathbb Z /
2\mathbb Z$, $a_{01}^1 =a_{10}^1= -2$. Let $\mathfrak g$ (resp.
$\hat{\mathfrak g}$) be the corresponding simple Lie algebra (resp.
affine Lie algebra ) associated to the Cartan matrix
$(a_{ij}^1)_{1\leq i, j\leq N}$ (resp. $A$). Note that the lattice
$\Rz$ is even in this case.

We define the normal ordered product of vertex operators
as follows.
\begin{align*}
&:Y^+(\g_i, a, b, z)Y^+(\g_j, a', b', w):\\
=&H_+(\g_i, z)H(\g_j,w)S(H_+(\g_i\otimes r^as^b, z^{-1})^*
H_+(\g_j\otimes r^{a'}s^{b'}, w^{-1})^*)\\
&\times e^{\g_i+\g_j}z^{\partial_{\g_i}}
w^{\partial_{\g_j}},\\
&:Y^+(\g_i, a, b, z)Y^-(\g_j, a', b', w):\\
=&H_+(\g_i, z)H(-\g_j\otimes r^{b'}s^{a'}, w)S(H_+(\g_i\otimes
r^as^b, z^{-1})^*
H_+(-\g_j\otimes r^{a'}s^{b'}, w^{-1})^*)\\
&\times e^{\g_i-\g_j}z^{\partial_{\g_i}}
w^{-\partial_{\g_j}}.
\end{align*}
Other normal ordered products are defined similarly.

The identities in the following theorems are understood as usual by
means of correlation functions (cf. e.g. \cite{FJ, J1, HZ, Z}).

\begin{theorem}  \label{th_ope}
Let $\xi=\g_0\otimes
((rs^{-1})^{\frac{1}{2}}+(r^{-1}s)^{\frac{1}{2}})-\pi\otimes
1_{\Cs}$.
  Then the  vertex operators
$Y^{\pm}( \g_i, a, b, z), Y^{\pm}(-\g_j, a, b, z)$,
 $\g_i\in  \G^*, k\in\mathbb Z$
acting on the group theoretically defined Fock space
$\FGC$
 satisfy the following relations.
\begin{eqnarray*}
 && Y^{+}(\g_i, a, b, s^{-b}z) Y^{+}(\g_j, a, b, s^{-b}w)\\
&&\qquad= \ep (\g_i, \g_j)
  :Y^{+}(\g_i, a, b, s^{-b}z) Y^{+}(\g_j, a, b, s^{-b}w):\\
&&\qquad\times\left\{
  \begin{array}{cc}
1& \mbox{ $\langle \g_i, \g_j\rangle_{\xi}^1=0$}\\
(z-r^{-a}s^{-b}w)^{-1}& \mbox{ $\langle \g_i, \g_j\rangle_{\xi}^1=-1$}\\
(z-r^{-a}s^{-b}(rs^{-1})^{\frac{1}{2}}w)(z-r^{-a}s^{-b}(r^{-1}s)^{\frac{1}{2}}w)&
\mbox{ $\langle \g_i, \g_j\rangle_{\xi}^1=2$}
\end{array},
\right.
 \end{eqnarray*}

\begin{eqnarray*}
 && Y^{-}(\g_i, a, b, r^{-a}z) Y^{-}(\g_j, a, b, r^{-a}w)\\
&&\qquad= \ep (\g_i, \g_j)^{-1}
  :Y^{-}(\g_i, a, b, r^{-a}z) Y^{-}(\g_j, a, b, r^{-a}w):\\
&&\qquad\times\left\{
  \begin{array}{cc}
1& \mbox{ $\langle \g_i, \g_j\rangle_{\xi}^1=0$}\\
(z-r^{a}s^{b}w)^{-1}& \mbox{ $\langle \g_i, \g_j\rangle_{\xi}^1=-1$}\\
(z-r^{a}s^{b}(rs^{-1})^{\frac{1}{2}}w)(z-r^{a}s^{b}(r^{-1}s)^{\frac{1}{2}}w)&
\mbox{ $\langle \g_i, \g_j\rangle_{\xi}^1=2$}
\end{array},
\right.
 \end{eqnarray*}

\begin{eqnarray*}
 && Y^{+}(\g_i, a, b, s^{-b}z) Y^{-}(\g_j, a, b,  r^{-a}w) \\
&&\qquad= \ep (\g_i, \g_j)
  :Y^{+}(\g_i, a, b, s^{-b}z) Y^{-}(\g_j, a, b, r^{-a}w):\\
&&\qquad\times\left\{
  \begin{array}{cc}
1& \mbox{ $\langle \g_i, \g_j\rangle_{\xi}^1=0$}\\
(z-r^{-a}s^{b}w) & \mbox{ $\langle \g_i, \g_j\rangle_{\xi}^1=-1$}\\
(z-r^{-a}s^{b}(rs^{-1})^{\frac{1}{2}}w)(z-r^{-a}s^{b}(r^{-1}s)^{\frac{1}{2}}w)&
\mbox{ $\langle \g_i, \g_j\rangle_{\xi}^1=2$}
\end{array},
\right.
 \end{eqnarray*}

\begin{eqnarray*}
 && Y^{-}(\g_i, a, b, r^{-a}z) Y^{+}(\g_j, a, b,  s^{-b}w) \\
&&\qquad= \ep (\g_i, \g_j)^{-1}
  :Y^{-}(\g_i, a, b, r^{-a}z) Y^{+}(\g_j, a, b, s^{-b}w):\\
&&\qquad\times\left\{
  \begin{array}{cc}
1& \mbox{ $\langle \g_i, \g_j\rangle_{\xi}^1=0$}\\
(z-r^{a}s^{-b}w) & \mbox{ $\langle \g_i, \g_j\rangle_{\xi}^1=-1$}\\
(z-r^{a}s^{-b}(rs^{-1})^{\frac{1}{2}}w)(z-r^{a}s^{-b}(r^{-1}s)^{\frac{1}{2}}w)&
\mbox{ $\langle \g_i, \g_j\rangle_{\xi}^1=2$}
\end{array},
\right.
 \end{eqnarray*}

\begin{eqnarray*}
 && Y^{+}(-\g_i, -a, -b, s^{-b}z) Y^{+}(-\g_j, -a, -b, s^{-b}w)\\
&&\qquad= \ep (\g_i, \g_j)
  :Y^{+}(-\g_i, -a, -b, s^{-b}z) Y^{+}(-\g_j, -a, -b, s^{-b}w):\\
&&\qquad\times\left\{
  \begin{array}{cc}
1& \mbox{ $\langle \g_i, \g_j\rangle_{\xi}^1=0$}\\
(z-r^{-a}s^{-b}w)^{-1}& \mbox{ $\langle \g_i, \g_j\rangle_{\xi}^1=-1$}\\
(z-r^{-a}s^{-b}(rs^{-1})^{\frac{1}{2}}w)(z-r^{-a}s^{-b}(r^{-1}s)^{\frac{1}{2}}w)&
\mbox{ $\langle \g_i, \g_j\rangle_{\xi}^1=2$}
\end{array},
\right.
 \end{eqnarray*}

\begin{eqnarray*}
 && Y^{-}(-\g_i, -a, -b, r^{-a}z) Y^{+}(-\g_j, -a, -b, r^{-a}w)\\
&&\qquad= \ep (\g_i, \g_j)^{-1}
  :Y^{-}(-\g_i, -a, -b, r^{-a}z) Y^{+}(-\g_j, -a, -b, r^{-a}w):\\
&&\qquad\times\left\{
  \begin{array}{cc}
1& \mbox{ $\langle \g_i, \g_j\rangle_{\xi}^1=0$}\\
(z-r^{a}s^{b}w)^{-1}& \mbox{ $\langle \g_i, \g_j\rangle_{\xi}^1=-1$}\\
(z-r^{a}s^{b}(rs^{-1})^{\frac{1}{2}}w)(z-r^{a}s^{b}(r^{-1}s)^{\frac{1}{2}}w)&
\mbox{ $\langle \g_i, \g_j\rangle_{\xi}^1=2$}
\end{array},
\right.
 \end{eqnarray*}

\begin{eqnarray*}
 && Y^{+}(-\g_i, -a, -b, s^{-b}z) Y^{-}(-\g_j, -a, -b,  r^{-a}w) \\
&&\qquad= \ep (\g_i, \g_j)
  :Y^{+}(-\g_i, -a, -b, s^{-b}z) Y^{-}(-\g_j, -a, -b, r^{-a}w):\\
&&\qquad\times\left\{
  \begin{array}{cc}
1& \mbox{ $\langle \g_i, \g_j\rangle_{\xi}^1=0$}\\
(z-r^{-a}s^{b}w) & \mbox{ $\langle \g_i, \g_j\rangle_{\xi}^1=-1$}\\
(z-r^{-a}s^{b}(rs^{-1})^{\frac{1}{2}}w)(z-r^{-a}s^{b}(r^{-1}s)^{\frac{1}{2}}w)&
\mbox{ $\langle \g_i, \g_j\rangle_{\xi}^1=2$}
\end{array},
\right.
 \end{eqnarray*}

\begin{eqnarray*}
 && Y^{-}(-\g_i, -a, -b, r^{-a}z) Y^{+}(-\g_j, -a, -b,  s^{-b}w) \\
&&\qquad= \ep (\g_i, \g_j)^{-1}
  :Y^{-}(-\g_i, -a, -b, r^{-a}z) Y^{+}(-\g_j, -a, -b, s^{-b}w):\\
&&\qquad\times\left\{
  \begin{array}{cc}
1& \mbox{ $\langle \g_i, \g_j\rangle_{\xi}^1=0$}\\
(z-r^{a}s^{-b}w) & \mbox{ $\langle \g_i, \g_j\rangle_{\xi}^1=-1$}\\
(z-r^{a}s^{-b}(rs^{-1})^{\frac{1}{2}}w)(z-r^{a}s^{-b}(r^{-1}s)^{\frac{1}{2}}w)&
\mbox{ $\langle \g_i, \g_j\rangle_{\xi}^1=2$}
\end{array},
\right.
 \end{eqnarray*}

\end{theorem}

\begin{remark} Replacing the vertex operator $Y^{\pm}$ by
$X^{\pm}$ via the characteristic map $ch$ in the above formulas, we
get the corresponding formulas for vertex operators $X^{\pm}(\g, a,
b, z)$ acting on $\VGC$.
\end{remark}

Now we consider the second distinguished class function
$$\wt^{r,\,s,\,\kappa}
=\g_0\otimes
((rs^{-1})^{\frac{1}{2}}+(r^{-1}s)^{\frac{1}{2}})-(\g_1\otimes
\kappa+\g_{r}\otimes \kappa^{-1}),$$ when $\G$ is a cyclic group of
order $N+1$.

In this case the Heisenberg algebra (\ref{E:heisen}) has the
following relations according to Prop. \ref{prop_orth}:
\begin{equation}\label{E:heisen1}
[a_m(\g_i), a_n(\g_j)]=
\begin{cases}
m\delta_{m, -n}((rs^{-1})^{\frac{m}{2}}+(r^{-1}s)^{\frac{m}{2}})\kappa^{mb_{ij}}C, & i=j\\
m\delta_{m, -n}a_{ij}^1\kappa^{mb_{ij}}C, & i\neq j
\end{cases},
\end{equation}
where $a_{ij}^1$ are the entries of the affine Cartan matrix of type
A and $r\geq 2$. This is the same Heisenberg subalgebra ($c=1$) in
$U_{r,s}(\widehat{\widehat{\mathfrak g}})$ provided that we identify
$$a_i(n)=\frac{[n]}na_n(\g_i).
$$

Recall that $(b_{ij})$ is the skew-symmetric matrix. We need to
slightly modify the definition of the middle term in the vertex
operators. For each $i=0, 1, \ldots, N$ we define the modified
operator $z^{\partial_{\g, \kappa}}$ on the group algebra $\mathbb
C[\Rz]$ by
\begin{equation}
z^{\partial_{\g_i, \kappa}}e^{\beta} =z^{\langle\g_i,
\beta\rangle_{\xi}^1} \kappa^{-\frac 12\sum_{j=1}^r\langle \g_i,
m_j\g_j\rangle_{\xi}^1b_{ij}}e^{\beta},
\end{equation}
where $\beta=\sum_{j}m_j\g_j\in \Rz$.

We then replace the operator $z^{\pm\partial_{\g_i}}$ in the
definition of the vertex operators $Y^{\pm}(\g_i, a, b, z)$ by the
operator $z^{\pm\partial_{\g_i, \kappa}}$. The formulas in Theorems
\ref{T:vertexop} remain true after the term $z^{\pm \partial}$
appearing in the formulas are modified accordingly.

The proof of the following theorem is similar to that of Theorem
\ref{th_ope}.

\begin{theorem}  \label{th_ope1}
Let $\G$ be a cyclic group of order $r+1$ and let
$\wt^{r,\,s,\,\kappa} =\g_0\otimes
((rs^{-1})^{\frac{1}{2}}+(r^{-1}s)^{\frac{1}{2}})-(\g_1\otimes
\kappa+\g_{r}\otimes \kappa^{-1}),$.
  The  vertex operators
$Y^{\pm}( \g_i, a, b, z)$ and $Y^{\pm}(-\g_i, a, b,  z),
 \g_i\in  \G^*$ acting on the group theoretically defined Fock space
$\FGC$
 satisfy the following relations.

\begin{eqnarray*}
 && Y^{+}(\g_i, a, b, s^{-b}z) Y^{+}(\g_j, a, b, s^{-b}w)\\
&&\qquad= \ep (\g_i, \g_j)
  :Y^{+}(\g_i, a, b, s^{-b}z) Y^{+}(\g_j, a, b, s^{-b}w):\\
&&\qquad\times\left\{
  \begin{array}{cc}
1& \mbox{ $\langle \g_i, \g_j\rangle_{\xi}^1=0$}\\
\kappa^{-\frac{1}{2}b_{ij}}(z-r^{-a}s^{-b}\kappa^{b_{ij}}w)^{-1}& \mbox{ $\langle \g_i, \g_j\rangle_{\xi}^1=-1$}\\
(z-r^{-a}s^{-b}(rs^{-1})^{\frac{1}{2}}w)(z-r^{-a}s^{-b}(r^{-1}s)^{\frac{1}{2}}w)&
\mbox{ $\langle \g_i, \g_j\rangle_{\xi}^1=2$}
\end{array},
\right.
 \end{eqnarray*}

\begin{eqnarray*}
 && Y^{-}(\g_i, a, b, r^{-a}z) Y^{-}(\g_j, a, b, r^{-a}w)\\
&&\qquad= \ep (\g_i, \g_j)^{-1}
  :Y^{-}(\g_i, a, b, r^{-a}z) Y^{-}(\g_j, a, b, r^{-a}w):\\
&&\qquad\times\left\{
  \begin{array}{cc}
1& \mbox{ $\langle \g_i, \g_j\rangle_{\xi}^1=0$}\\
\kappa^{-\frac{1}{2}b_{ij}}(z-r^{a}s^{b}\kappa^{b_{ij}}w)^{-1}& \mbox{ $\langle \g_i, \g_j\rangle_{\xi}^1=-1$}\\
(z-r^{a}s^{b}(rs^{-1})^{\frac{1}{2}}w)(z-r^{a}s^{b}(r^{-1}s)^{\frac{1}{2}}w)&
\mbox{ $\langle \g_i, \g_j\rangle_{\xi}^1=2$}
\end{array},
\right.
 \end{eqnarray*}

\begin{eqnarray*}
 && Y^{+}(\g_i, a, b, s^{-b}z) Y^{-}(\g_j, a, b,  r^{-a}w) \\
&&\qquad= \ep (\g_i, \g_j)
  :Y^{+}(\g_i, a, b, s^{-b}z) Y^{-}(\g_j, a, b, r^{-a}w):\\
&&\qquad\times\left\{
  \begin{array}{cc}
1& \mbox{ $\langle \g_i, \g_j\rangle_{\xi}^1=0$}\\
\kappa^{-\frac{1}{2}b_{ij}}(z-r^{-a}s^{b}\kappa^{b_{ij}}w) & \mbox{ $\langle \g_i, \g_j\rangle_{\xi}^1=-1$}\\
(z-r^{-a}s^{b}(rs^{-1})^{\frac{1}{2}}w)(z-r^{-a}s^{b}(r^{-1}s)^{\frac{1}{2}}w)&
\mbox{ $\langle \g_i, \g_j\rangle_{\xi}^1=2$}
\end{array},
\right.
 \end{eqnarray*}

\begin{eqnarray*}
 && Y^{-}(\g_i, a, b, r^{-a}z) Y^{+}(\g_j, a, b,  s^{-b}w) \\
&&\qquad= \ep (\g_i, \g_j)^{-1}
  :Y^{-}(\g_i, a, b, r^{-a}z) Y^{+}(\g_j, a, b, s^{-b}w):\\
&&\qquad\times\left\{
  \begin{array}{cc}
1& \mbox{ $\langle \g_i, \g_j\rangle_{\xi}^1=0$}\\
\kappa^{-\frac{1}{2}b_{ij}}(z-r^{a}s^{-b}\kappa^{b_{ij}}w) & \mbox{ $\langle \g_i, \g_j\rangle_{\xi}^1=-1$}\\
(z-r^{a}s^{-b}(rs^{-1})^{\frac{1}{2}}w)(z-r^{a}s^{-b}(r^{-1}s)^{\frac{1}{2}}w)&
\mbox{ $\langle \g_i, \g_j\rangle_{\xi}^1=2$}
\end{array},
\right.
 \end{eqnarray*}

\begin{eqnarray*}
 && Y^{+}(-\g_i, -a, -b, s^{-b}z) Y^{+}(-\g_j, -a, -b, s^{-b}w)\\
&&\qquad= \ep (\g_i, \g_j)
  :Y^{+}(-\g_i, -a, -b, s^{-b}z) Y^{+}(-\g_j, -a, -b, s^{-b}w):\\
&&\qquad\times\left\{
  \begin{array}{cc}
1& \mbox{ $\langle \g_i, \g_j\rangle_{\xi}^1=0$}\\
\kappa^{-\frac{1}{2}b_{ij}}(z-r^{-a}s^{-b}\kappa^{b_{ij}}w)^{-1}& \mbox{ $\langle \g_i, \g_j\rangle_{\xi}^1=-1$}\\
(z-r^{-a}s^{-b}(rs^{-1})^{\frac{1}{2}}w)(z-r^{-a}s^{-b}(r^{-1}s)^{\frac{1}{2}}w)&
\mbox{ $\langle \g_i, \g_j\rangle_{\xi}^1=2$}
\end{array},
\right.
 \end{eqnarray*}

\begin{eqnarray*}
 && Y^{-}(-\g_i, -a, -b, r^{-a}z) Y^{+}(-\g_j, -a, -b, r^{-a}w)\\
&&\qquad= \ep (\g_i, \g_j)^{-1}
  :Y^{-}(-\g_i, -a, -b, r^{-a}z) Y^{+}(-\g_j, -a, -b, r^{-a}w):\\
&&\qquad\times\left\{
  \begin{array}{cc}
1& \mbox{ $\langle \g_i, \g_j\rangle_{\xi}^1=0$}\\
\kappa^{-\frac{1}{2}b_{ij}}(z-r^{a}s^{b}\kappa^{b_{ij}}w)^{-1}& \mbox{ $\langle \g_i, \g_j\rangle_{\xi}^1=-1$}\\
(z-r^{a}s^{b}(rs^{-1})^{\frac{1}{2}}w)(z-r^{a}s^{b}(r^{-1}s)^{\frac{1}{2}}w)&
\mbox{ $\langle \g_i, \g_j\rangle_{\xi}^1=2$}
\end{array},
\right.
 \end{eqnarray*}

\begin{eqnarray*}
 && Y^{+}(-\g_i, -a, -b, s^{-b}z) Y^{-}(-\g_j, -a, -b,  r^{-a}w) \\
&&\qquad= \ep (\g_i, \g_j)
  :Y^{+}(-\g_i, -a, -b, s^{-b}z) Y^{-}(-\g_j, -a, -b, r^{-a}w):\\
&&\qquad\times\left\{
  \begin{array}{cc}
1& \mbox{ $\langle \g_i, \g_j\rangle_{\xi}^1=0$}\\
\kappa^{-\frac{1}{2}b_{ij}}(z-r^{-a}s^{b}\kappa^{b_{ij}}w) & \mbox{ $\langle \g_i, \g_j\rangle_{\xi}^1=-1$}\\
(z-r^{-a}s^{b}(rs^{-1})^{\frac{1}{2}}w)(z-r^{-a}s^{b}(r^{-1}s)^{\frac{1}{2}}w)&
\mbox{ $\langle \g_i, \g_j\rangle_{\xi}^1=2$}
\end{array},
\right.
 \end{eqnarray*}

\begin{eqnarray*}
 && Y^{-}(-\g_i, -a, -b, r^{-a}z) Y^{+}(-\g_j, -a, -b,  s^{-b}w) \\
&&\qquad= \ep (\g_i, \g_j)^{-1}
  :Y^{-}(-\g_i, -a, -b, r^{-a}z) Y^{+}(-\g_j, -a, -b, s^{-b}w):\\
&&\qquad\times\left\{
  \begin{array}{cc}
1& \mbox{ $\langle \g_i, \g_j\rangle_{\xi}^1=0$}\\
\kappa^{-\frac{1}{2}b_{ij}}(z-r^{a}s^{-b}\kappa^{b_{ij}}w) & \mbox{ $\langle \g_i, \g_j\rangle_{\xi}^1=-1$}\\
(z-r^{a}s^{-b}(rs^{-1})^{\frac{1}{2}}w)(z-r^{a}s^{-b}(r^{-1}s)^{\frac{1}{2}}w)&
\mbox{ $\langle \g_i, \g_j\rangle_{\xi}^1=2$}
\end{array},
\right.
 \end{eqnarray*}

\end{theorem}

\begin{remark}
Replacing the vertex operators $Y^{\pm}$ by $X^{\pm}$ via the
characteristic map $ch$ we obtain the corresponding results on the
space $\VGC$.
\end{remark}

\subsection{Quantum vertex representations of $U_{r,\,s}(\widehat{\widehat{g}})$}
For each $i=0$, $\dots$, $N$ let
\begin{equation*}
\widetilde{a_i}(n)=\frac{[n]}n a_n(\g_i).
\end{equation*}
It follows from (\ref{E:heisen}) and (\ref{E:heisen1}) that
\begin{equation}\label{E:heisenberg2}
[\widetilde{a_i}(m), \widetilde{a_j}(n)]=\delta_{m,
-n}\frac{(rs)^{\frac{m(1-a_{ij})}{2}}[m\langle \g_i,
\g_j\rangle_{\xi}^1]}m [m].
\end{equation}
According to McKay, the bilinear form $\langle \g_i,
\g_j\rangle_{\xi}^1$ is exactly the same as the invariant form $(\
|\ )$ of the root lattice of the affine Lie algebra $\loopg$. This
implies that the commutation relations (\textrm{T2})
 are exactly the commutation relations (\ref{E:heisenberg2})
 of the Heisenberg
algebra in $U_{r,\,s}(\widehat{\widehat{g}})$ if we identify
$\widetilde{a_i}(n)$ with ${a_i}(n)$. Thus the Fock space $\SGC$ is
a level one representation for the Heisenberg subalgebra in
$U_{r,\,s}(\widehat{\widehat{g}})$.

The following theorem gives a $q$-deformation of the new form of
McKay correspondence in \cite{FJW} and provides a direct connection
from a finite subgroup $\G$ of $SU_2$ to the quantum toroidal
algebra $U_{r,\,s}(\widehat{\widehat{g}})$ of $ADE$ type.

\begin{theo}  \label{T:quantum} Given a finite subgroup $\G$ of $SU_2$,
each of the following correspondence gives a vertex representation
of the quantum toroidal algebra $U_{r,\,s}(\widehat{\widehat{g}})$
on the Fock space $\FGC$:
\begin{align*}
x_i^{+}(n)& \longrightarrow Y^{+}_n(\g_i\otimes s^{b}, \frac{1}{2}, -\frac{1}{2}), \\
x_i^{-}(n)& \longrightarrow Y^{-}_n(\g_i\otimes r^{a}, \frac{1}{2}, -\frac{1}{2}), \\
a_i(m) &\longrightarrow \frac{[m]}m a_m(\g_i), \qquad \hbox{for}
\quad
m>0; \\
a_i(m) &\longrightarrow \frac{-[-m]}m a_m(\g_i), \qquad
\hbox{for}\quad  m<
0;\\
 \gamma &\longrightarrow r, \qquad \gamma'\longrightarrow s;
\end{align*}
or
\begin{align*}
x_i^{+}(n)& \longrightarrow Y^{-}_n(-\g_i\otimes s^{b}, -\frac{1}{2}, \frac{1}{2}), \\
x_i^{-}(n)& \longrightarrow Y^{+}_n(-\g_i\otimes r^{a}, -\frac{1}{2}, \frac{1}{2}), \\
a_i(m) &\longrightarrow \frac{[m]}m a_m(\g_i), \qquad \hbox{for}
\quad
m>0; \\
a_i(m) &\longrightarrow \frac{-[-m]}m a_m(\g_i), \qquad
\hbox{for}\quad  m<
0;\\
 \gamma &\longrightarrow r, \qquad \gamma' \longrightarrow s;
\end{align*}
where $i=0, \dots, N$, and $m\in \mathbb{Z}$.
\end{theo}

\begin{remark}  Replacing $Y^{\pm}$ by $X^{\pm}$ in the above theorem,
we obtain a vertex representation of
$U_{r,\,s}(\widehat{\widehat{g}})$ in the space $\VGC$.
\end{remark}

According to McKay, the bilinear form $\langle \g_i,
\g_j\rangle_{\xi}^1$ is exactly the same as the invariant form $(\
|\ )$ of the root lattice of the affine Lie algebra $\loopg$.  Thus
the Fock space $\SGC$ is a level one representation for the
Heisenberg subalgebra in ${\mathcal U}_{r,s}(\widehat{\frak
{sl}_n})$.

Denote by $\SGGC$ the symmetric algebra generated
 by
$a_{-n} (\g_i)$, $n >0$, $i =1, \ldots , N$ over $\mathbb C[r^{\pm
1}, s^{\pm1}]$. $\SGGC$ is isometric to $\RGGC$.

We define
$$
 \FGGC  = \RGGC \otimes \mathbb C [ \Rzz]
 \cong \SGGC\otimes \mathbb C [ \Rzz].
$$
 The space $\VGC$ associated to the lattice
$\Rz$ is isomorphic to the tensor
product of the space $\RGGC$
and $\Rz$ as well as the space associated to
the rank $1$ lattice $\mathbb Z \alpha_0$.

The following theorem gives the new form of McKay correspondence in
\cite{FJW2} for two-parameter case and provides a direct connection
from a finite subgroup $\G$ of $SU_2$ to the two-parameter quantum
affine algebra ${\mathcal U}_{r,s}(\widehat{\frak {sl}_n})$.

\begin{theo} Given a finite subgroup $\G$ of $SU_2$, each of
the following correspondence gives the basic representation of the
two-parameter quantum affine algebra ${\mathcal
U}_{r,s}(\widehat{\frak {sl}_n})$ on the Fock space $\FGGC$:
\begin{align*}
x_i^{+}(n)& \longrightarrow Y^{+}_n(\g_i\otimes s^{b}, \frac{1}{2}, -\frac{1}{2}), \\
x_i^{-}(n)& \longrightarrow Y^{-}_n(\g_i\otimes r^{a}, \frac{1}{2}, -\frac{1}{2}), \\
a_i(m) &\longrightarrow \frac{[m]}m a_m(\g_i), \qquad \hbox{for}
\quad
m>0; \\
a_i(m) &\longrightarrow \frac{-[-m]}m a_m(\g_i), \qquad
\hbox{for}\quad  m<
0;\\
 \gamma &\longrightarrow r, \qquad \gamma'\longrightarrow s;
\end{align*}
or
\begin{align*}
x_i^{+}(n)& \longrightarrow Y^{-}_n(-\g_i\otimes s^{b}, -\frac{1}{2}, \frac{1}{2}), \\
x_i^{-}(n)& \longrightarrow Y^{+}_n(-\g_i\otimes r^{a}, -\frac{1}{2}, \frac{1}{2}), \\
a_i(m) &\longrightarrow \frac{[m]}m a_m(\g_i), \qquad \hbox{for}
\quad
m>0; \\
a_i(m) &\longrightarrow \frac{-[-m]}m a_m(\g_i), \qquad
\hbox{for}\quad  m<
0;\\
 \gamma &\longrightarrow r, \qquad \gamma' \longrightarrow s;
\end{align*}
where $i=1, \dots, N$.
\end{theo}

\begin{remark}  Replacing $Y^{\pm}$ by $X^{\pm}$ in the above theorem,
we obtain a vertex representation of ${\mathcal
U}_{r,s}(\widehat{\frak {sl}_n})$ in the space $\VGC$.
\end{remark}

\vskip30pt \centerline{\bf ACKNOWLEDGMENT}

\bigskip

N. Jing would like to thank the support of NSA grant and NSFC Grant
(No. 10728102). H. Zhang would like to thank the support of NSFC
(No. 10801094) and Shanghai Leading Academic Discipline Project (No.
J50101).

\bigskip

\bibliographystyle{amsalpha}

\end{document}